\documentclass[12pt]{amsart}
\usepackage{epsfig,color}
\usepackage{blindtext}
\usepackage{graphicx}
\usepackage{enumitem}
\usepackage{url}
\usepackage{amssymb}
\usepackage{graphicx,import}
\usepackage{comment}
\usepackage{esint}
\usepackage{xcolor}
\usepackage{mathtools}
\usepackage{comment}
\usepackage{tikz}
\usepackage{mathrsfs}
\usepackage{hyperref}

\usepackage[margin = 1in] {geometry}

\usepackage{hyperref}
\usepackage{dsfont}

\newtheorem{theorem}{Theorem}[section]

\newtheorem{proposition}[theorem]{Proposition}
\newtheorem{lemma}[theorem]{Lemma}
\newtheorem{corollary}[theorem]{Corollary}
\newtheorem{question}[theorem]{Question}

\theoremstyle{definition}
\newtheorem{example}[theorem]{Example}
\newtheorem{definition}[theorem]{Definition}
\newtheorem*{remark*}{Remark}
\newtheorem*{acknowledgements}{Acknowledgements}
\newtheorem{remark}[theorem]{Remark}

\numberwithin{equation}{section}

\newcommand{\cA}{\mathcal{A}}

\newcommand{\cH}{\mathcal{H}}
\newcommand{\R}{\mathbb{R}}

\newcommand{\N}{\mathbb{N}}

\newcommand{\eps}{\varepsilon}

\newcommand{\cCEE}{\mathcal{C}_{E_+,E_-}}

\newcommand{\Ric}{\mathrm{Ric}}

\newcommand{\interior}{\mathrm{Int}}
\newcommand{\intrel}{\mathrm{int}_{\mathrm{rel}}}

\newcommand{\Span}{\mathrm{Span}}
\newcommand{\Loc}{\mathrm{Loc}}

\title[Min-max construction of PMC hypersurfaces in noncompact manifolds]{Min-max construction of prescribed mean curvature hypersurfaces in noncompact manifolds}

\author{Douglas Stryker}

\begin{document}

\begin{abstract}
    We develop a min-max theory for hypersurfaces of prescribed mean curvature in noncompact manifolds, applicable to prescription functions that do not change sign outside a compact set. We use this theory to prove new existence results for closed prescribed mean curvature hypersurfaces in Euclidean space and complete finite area constant mean curvature hypersurfaces in finite volume manifolds.
\end{abstract}

\maketitle


\section{Introduction}
Given a smooth function $h$ on a Riemannian manifold $(M, g)$, a hypersurface $\Sigma \subset M$ has \emph{prescribed mean curvature} $h$ (abbreviated $h$-PMC) if the mean curvature of $\Sigma$ satisfies $H_{\Sigma} = h\mid_{\Sigma}$. The existence question for closed PMC hypersurfaces is a fundamental problem in geometric analysis. Moreover, PMC hypersurfaces are useful tools in Riemannian geometry; for instance, $\mu$-bubbles (which are special PMC hypersurfaces introduced by \cite{mu-bubbles}) have been instrumental in recent breakthroughs in the study of positive scalar curvature (for example, see \cite{Gromovaspherical,CLL,CLaspherical}).

A fruitful approach to the existence problem relies on the observation that $h$-PMC hypersurfaces arise as the boundaries of critical points of the functional
\[ \cA^h(\Omega) = \mathrm{Area}(\partial \Omega) - \int_{\Omega} h \]
for finite perimeter sets $\Omega$. Hence, $h$-PMC hypersurfaces can be constructed using minimization techniques (see for instance \cite{mu-bubbles,morgan,zhu:mububbles}) and min-max techniques (see for instance \cite{zhou-zhu_CMC,zhou-zhu_PMC, zhou_mult}) for $\cA^h$. These minimization and min-max results for $\cA^h$ build on the original minimization and min-max programs for the area functional; we refer the reader to \cite{FedFlem, Flem, DG, Alm66, simons} and \cite{almgren1962homotopy,pitts,MN_morse,MN_morse_mult} for some particularly relevant results in these directions.

An essential difficulty arises for these variational techniques in noncompact ambient spaces, since critical points are constructed as subsequential limits of minimizing or min-max sequences, which may escape to infinity in a noncompact manifold. Despite these difficulties, there are many important problems related to the existence of PMC hypersurfaces in noncompact manifolds.

A first example is a well-known question from \cite[Problem 59]{yauQuestions}: for which functions $h : \R^3 \to \R$ does there exist an embedded $h$-PMC surface with prescribed genus? For some progress towards this problem in the case of genus zero, we refer the reader to \cite{BK, TW, yauremark}. It is natural to consider the following modification of Yau's problem, where we remove the restriction that the surface has prescribed genus and weaken the restriction of embedding to almost embedding\footnote{A codimension 1 immersion is an almost embedding if it can be decomposed locally near any point into a union of embedded sheets that may touch each other but do not cross through each other.}.

\begin{question}\label{quest:PMCrn}
    For which locally smooth functions $h : \R^{n+1} \to \R$ does there exist a closed almost embedded $h$-PMC hypersurface?
\end{question}

A reason to remove the topological restriction from the variational point of view is the fact that minimization and min-max for $\cA^h$ over surfaces of a fixed topological type does not produce smooth solutions in general (see \cite{SS} and \cite{4spheres}). Moreover, it seems difficult or impossible to rule out hypersurfaces that are almost embedded but not embedded in general (see \cite{zhou-zhu_CMC, zhou-zhu_PMC}), even though embeddedness can be guaranteed in some special situations (see \cite{white_transverse, zhou_mult, Bel-Work}). Progress towards Question \ref{quest:PMCrn} using min-max methods was obtained by \cite{liam}.

A second example is a natural synthesis of the main theorem of \cite{CL}---every complete finite volume manifold contains a complete finite area embedded minimal hypersurface---and the main theorem of \cite{zhou-zhu_CMC}---every closed manifold contains a closed almost embedded hypersurface of constant mean curvature $c$ (abbreviated $c$-CMC) for all $c > 0$.

\begin{question}\label{quest:CMCfv}
    Does every complete finite volume manifold contain a complete finite area almost embedded $c$-CMC hypersurface for every $c > 0$?
\end{question}

We refer the reader to \cite{thorbergsson, bangert, song_dich, CL, localization, dey} for results in the case $c=0$. The case of nonzero CMC hypersurfaces appears to be mostly unexplored. Even a positive answer for an inexplicit range of small values of $c > 0$ via an implicit function theorem argument (using as a starting point the existence of a finite area embedded minimal hypersurface) seems difficult, since such hypersurfaces may be noncompact in general.

\subsection{Main results}

In this paper, we develop a localized min-max theory for complete finite area PMC hypersurfaces in complete noncompact manifolds. The following statement is a simplified version of our main result, where we postpone the precise min-max definitions to \S\ref{subsec:min-max_setup} (see Theorem \ref{thm:localized_noncompact} for the most general result). For clarity, when we say there is a homotopy class of sweepouts with nontrivial width for $\cA^h$, we mean that we are in a situation where Morse theory (or mountain pass) heuristics suggest the existence of a critical point of the functional.

\begin{theorem}\label{thm:main}
    Let $(M, g)$ be a complete manifold of dimension $3 \leq n+1 \leq 7$. Let $h \in C^{\infty}_{\text{loc}}(M)$ be nonpositive outside a compact set. If there is a homotopy class of sweepouts of bounded finite perimeter sets with nontrivial width for $\cA^h$, then there is a smooth, complete, finite area, almost embedded $h$-PMC hypersurface in $M$.
\end{theorem}

Theorem \ref{thm:main} is the first general existence result for PMC hypersurfaces in arbitrary complete noncompact manifolds (meaning we do not require $(M, g)$ to have special structure).

Surprisingly, the fact that Theorem \ref{thm:main} works (with the sign restriction on the prescribing function $h$) boils down to the fact that negative real numbers $a$ satisfy $-a > a$ (and nonnegative numbers do not). To see exactly where this simple observation enters the proof, we refer the reader to the outline of the proof below (see equation \eqref{eqn:intro}).

In the course of the proof of Theorem \ref{thm:main}, we require a min-max theorem for compact manifolds whose boundary is a good barrier. The following statement is a simplified version (see Theorem \ref{thm:localized_compact} for the most general result).

\begin{theorem}\label{thm:main_compact}
    Let $(N, \partial N, g)$ be a compact manifold of dimension $3 \leq n+1 \leq 7$ with smooth boundary. Let $h \in C^{\infty}(N)$ satisfy $h\mid_{\partial N} < H_{\partial N}$. If there is a homotopy class of sweepouts of finite perimeter sets in $\interior(N)$ with nontrivial width for $\cA^h$, then there is a smooth, closed, almost embedded $h$-PMC hypersurface in $\interior(N)$.
\end{theorem}

The main novelty of Theorem \ref{thm:main_compact} is that we only require a \emph{one-sided} bound on $h$ relative to the mean curvature of the boundary. This setting is more delicate because the possibility of boundary touching depends on the orientation of the hypersurface. For instance, a minimal boundary cannot be touched from the inside by a $1$-CMC hypersurface oriented \emph{away} from the boundary, but it \emph{can} be touched from the inside by a $1$-CMC hypersurface oriented \emph{towards} the boundary (e.g.\ any sphere in Euclidean space has nonzero CMC but touches its tangent planes from one side).

We apply Theorems \ref{thm:main} and \ref{thm:main_compact} to Questions \ref{quest:PMCrn} and \ref{quest:CMCfv}. Towards Question \ref{quest:PMCrn}, we have the following result.

\begin{corollary}\label{cor:PMCrn}
    If $h \in C^{\infty}_{\text{loc}}(\R^{n+1})$ for $3 \leq n+1 \leq 7$ satisfies
    \begin{enumerate}
        \item there exists a nonempty bounded finite perimeter  set $\Lambda$ satisfying $\cA^h(\Lambda) \leq 0$,
        \item $h(x) \leq C|x|^{-1-\alpha}$ for some $\alpha,\ C > 0$,
    \end{enumerate}
    then there is a smooth, closed, almost embedded $h$-PMC hypersurface in $\R^{n+1}$.
\end{corollary}

We emphasize that condition (2) above is a one-sided bound, meaning that $h$ can be arbitrarily large in magnitude if it is negative.

Corollary \ref{cor:PMCrn} applies to a set of prescribing functions that is disjoint from the set of prescribing functions considered by \cite{liam}\footnote{Both results require $\cA^h$ to have a nontrivial mountain pass over bounded sets, but \cite{liam} requires $h$ to limit to a positive constant at infinity, whereas our result requires $h$ to be nonpositive at infinity.}. A nice feature of our result is that, unlike most of the earlier progress towards Question \ref{quest:PMCrn}, we do not require any restrictions on the derivatives of $h$, nor do we require $h$ to limit to a constant at infinity. Moreover, our technique is flexible enough to handle prescribing functions in any complete noncompact ambient space (and not just spaces with nice structure).

Towards Question \ref{quest:CMCfv}, we have the following result.

\begin{corollary}\label{cor:CMCfv}
    Let $(M, g)$ be a complete finite volume manifold of dimension $3 \leq n+1 \leq 7$. If
    \[
        0 < c < \mathscr{D}(M, g) := \sup_{v_1 \neq v_2} \left|\frac{\mathscr{I}(v_2) - \mathscr{I}(v_1)}{v_2 - v_1}\right|,
    \]
    where $\mathscr{I}$ is the isoperimetric profile of $(M, g)$, then there is a smooth, complete, finite area, almost embedded $c$-CMC hypersurface in $M$.
\end{corollary}

Since $\mathscr{D}(M, g) > 0$ for all complete finite volume manifolds, we give a positive answer to Question \ref{quest:CMCfv} for a nonempty range of $c$ for every complete finite volume manifold. We emphasize that this conclusion is the first positive result towards Question \ref{quest:CMCfv}. 

Moreover, we completely resolve Question \ref{quest:CMCfv} for any manifold with $\mathscr{D}(M, g) = +\infty$. We provide examples (see Example \ref{ex:rot_sym}) of complete noncompact manifolds with finite volume satisfying $\mathscr{D}(M, g) = +\infty$, so the class of such manifolds is nontrivial.

Finally, we obtain a clean application to complete finite volume hyperbolic manifolds.

\begin{corollary}\label{cor:finitevolhyp}
    Let $(M, g)$ be a complete finite volume hyperbolic manifold of dimension $3 \leq n+1 \leq 7$. For any $0 < c < n$, there is a smooth, complete, finite area, almost embedded $c$-CMC hypersurface in $M$.
\end{corollary}

\subsection{Outline of the proof of Theorem \ref{thm:main}}
We use a tubular neighborhood cutting scheme similar to the one developed in \cite{localization} for the area functional. In that work, the crucial fact is that if we delete a small tubular neighborhood of a two-sided unstable minimal hypersurface, the new boundary components will be strictly mean convex, and therefore provide a good barrier for subsequent min-max constructions. 
A key difficult for the $\cA^h$ functional is that if we delete a tubular neighborhood of a two-sided $\cA^h$-unstable $h$-PMC hypersurface, then one of the new boundary components may not provide a good barrier for subsequent min-max constructions. Instead, we develop a more delicate cutting argument.

Suppose $h \in C^{\infty}_{\text{loc}}(M)$ is negative outside a compact set $K$\footnote{The case where $h$ is only nonpositive outside $K$ follows by approximation.} and $\Phi_0 : [0, 1] \to \mathcal{C}(M)$ is a sweepout of bounded finite perimeter sets whose homotopy class has nontrivial width for $\cA^h$. By enlarging $K$ if necessary, suppose that $\overline{\Phi_0(t)} \subset K$ for all $t \in [0, 1]$.

We take an exhaustion of $M$ by compact sets $\{N_i\}$ with smooth boundary satisfying $K \subset N_i$ for all $i$. We modify $h$ in a small neighborhood of $\partial N_i$ so that $h\mid_{\partial N_i} < H_{\partial N_i}$, making Theorem \ref{thm:main_compact} applicable. We also perturb $h$ and the metric (to a ``good pair'' in the sense of \cite[\S3.3]{zhou_mult}) so that the hypersurface produced by Theorem \ref{thm:main_compact} is an embedded boundary that is unstable for $\cA^h$. We claim that we can find an $h$-PMC hypersurface intersecting $K$ (with uniform area and index bounds). If so, then the result follows by taking a limit of these hypersurfaces as $i \to \infty$, using the compactness theorem of \cite[Theorem 2.8]{zhou_mult} (which builds on the compactness results of \cite{sharp}). The limit cannot escape to infinity because it must intersect $K$.

If the hypersurface $\Sigma = \partial \Omega$ produced by Theorem \ref{thm:main_compact} intersects $K$, then we are done. Otherwise, we either have $K \subset \Omega$ or $K \cap \Omega = \varnothing$. In either case, we find a subdomain of $N_i$ that contains $K$ but does not contain $\Sigma$ and satisfies the hypotheses of Theorem \ref{thm:main_compact} (see the below paragraphs). So we produce a new $h$-PMC hypersurface that is distinct from $\Sigma$. By choosing $h$ and the metric generically (so that every closed $h$-PMC hypersurface is nondegenerate), we can only iterate these steps finitely many times, so we eventually find $\Sigma$ intersecting $K$.

Case 1: $K \subset \Omega$. Let $\Sigma^*$ be a component of $\partial \Omega$ that is unstable for $\cA^h$. By pushing $\Sigma^*$ into $\Omega$, the new boundary component $\tilde{\Sigma}^*$ satisfies $h\mid_{\tilde{\Sigma}^*} < H_{\tilde{\Sigma}^*}$. We can similarly modify the other components of $\partial \Omega$ to produce $\tilde{\Omega}$ satisfying $\Sigma^* \cap \tilde{\Omega} = \varnothing$ and $h\mid_{\partial \tilde{\Omega}} < H_{\partial \tilde{\Omega}}$. Hence, $\Sigma \not\subset \tilde{\Omega}$, and we can apply Theorem \ref{thm:main_compact} in $\tilde{\Omega}$ and iterate. For an illustration of this case, see Figure \ref{fig:1}.

\tikzset{every picture/.style={line width=0.75pt}} 

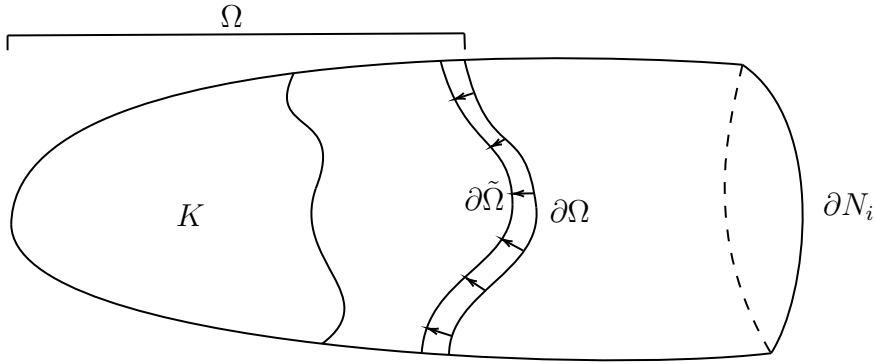
\begin{figure}
\begin{tikzpicture}[x=0.75pt,y=0.75pt,yscale=-0.8,xscale=0.8]

\draw    (92,150.5) .. controls (93,27.5) and (486,43.5) .. (553,49.5) ;
\draw    (92,150.5) .. controls (94,228.5) and (453,245.5) .. (571,231.5) ;
\draw    (553,49.5) .. controls (607,86.5) and (594,196.5) .. (571,231.5) ;
\draw  [dash pattern={on 4.5pt off 4.5pt}]  (553,49.5) .. controls (543,92.5) and (527,158.5) .. (571,231.5) ;
\draw    (284,127.5) .. controls (304,77.5) and (251,94.17) .. (270,55.17) ;
\draw    (288,225.5) .. controls (328,195.5) and (268,173.5) .. (284,127.5) ;
\draw    (423,139.5) .. controls (416,81.5) and (393.67,114.17) .. (377.67,46.17) ;
\draw    (368,232.5) .. controls (368,192.5) and (427,180.5) .. (423,139.5) ;
\draw [line width=0.75]    (369.8,220.6) -- (356.57,216.43) ;
\draw [shift={(354.67,215.83)}, rotate = 17.48] [color={rgb, 255:red, 0; green, 0; blue, 0 }  ][line width=0.75]    (6.56,-1.97) .. controls (4.17,-0.84) and (1.99,-0.18) .. (0,0) .. controls (1.99,0.18) and (4.17,0.84) .. (6.56,1.97)   ;
\draw [line width=0.75]    (390.6,191.8) -- (380.67,185.27) ;
\draw [shift={(379,184.17)}, rotate = 33.35] [color={rgb, 255:red, 0; green, 0; blue, 0 }  ][line width=0.75]    (6.56,-1.97) .. controls (4.17,-0.84) and (1.99,-0.18) .. (0,0) .. controls (1.99,0.18) and (4.17,0.84) .. (6.56,1.97)   ;
\draw [line width=0.75]    (414.6,166.67) -- (403.09,160.45) ;
\draw [shift={(401.33,159.5)}, rotate = 28.38] [color={rgb, 255:red, 0; green, 0; blue, 0 }  ][line width=0.75]    (6.56,-1.97) .. controls (4.17,-0.84) and (1.99,-0.18) .. (0,0) .. controls (1.99,0.18) and (4.17,0.84) .. (6.56,1.97)   ;
\draw [line width=0.75]    (421.4,130.6) -- (410.67,130.8) ;
\draw [shift={(408.67,130.83)}, rotate = 358.95] [color={rgb, 255:red, 0; green, 0; blue, 0 }  ][line width=0.75]    (6.56,-1.97) .. controls (4.17,-0.84) and (1.99,-0.18) .. (0,0) .. controls (1.99,0.18) and (4.17,0.84) .. (6.56,1.97)   ;
\draw [line width=0.75]    (403,96.6) -- (396.42,100.21) ;
\draw [shift={(394.67,101.17)}, rotate = 331.28] [color={rgb, 255:red, 0; green, 0; blue, 0 }  ][line width=0.75]    (6.56,-1.97) .. controls (4.17,-0.84) and (1.99,-0.18) .. (0,0) .. controls (1.99,0.18) and (4.17,0.84) .. (6.56,1.97)   ;
\draw [line width=0.75]    (383.4,67.4) -- (373.57,70.56) ;
\draw [shift={(371.67,71.17)}, rotate = 342.2] [color={rgb, 255:red, 0; green, 0; blue, 0 }  ][line width=0.75]    (6.56,-1.97) .. controls (4.17,-0.84) and (1.99,-0.18) .. (0,0) .. controls (1.99,0.18) and (4.17,0.84) .. (6.56,1.97)   ;
\draw    (350.8,230.9) .. controls (354.2,183.4) and (408.33,176.5) .. (408,137.5) ;
\draw    (408,137.5) .. controls (407.67,95.17) and (377.8,103.8) .. (362.6,47) ;
\draw    (89.67,31) -- (377.67,29.67) ;
\draw    (89.67,31) -- (89.67,40.33) ;
\draw    (377.67,29.67) -- (378,39.88) ;

\draw (602,128.4) node [anchor=north west][inner sep=0.75pt]    {$\partial N_{i}$};
\draw (194,135.4) node [anchor=north west][inner sep=0.75pt]    {$K$};
\draw (429.6,132.8) node [anchor=north west][inner sep=0.75pt]    {$\partial \Omega $};
\draw (375.8,121) node [anchor=north west][inner sep=0.75pt]    {$\partial \tilde{\Omega }$};
\draw (222,10.4) node [anchor=north west][inner sep=0.75pt]    {$\Omega $};

\end{tikzpicture}
\caption{$K \subset \Omega$.}
\label{fig:1}
\end{figure}

Case 2: $K \cap \Omega = \varnothing$. We let $\tilde{\Omega}$ be the open set obtained by slightly enlarging $\Omega$ along all its boundary components. Then $\tilde{N}_i = N_i \setminus \tilde{\Omega}$ contains $K$, does not contain $\Sigma$, and satisfies
\begin{equation}\label{eqn:intro}
H_{\partial \tilde{N}_i} \approx -H_{\partial \Omega} = -h\mid_{\partial \Omega} > h\mid_{\partial \Omega},
\end{equation}
where the last inequality follows from the fact that $h\mid_{N_i \setminus K} < 0$. Hence, we can apply Theorem \ref{thm:main_compact} in $\tilde{N}_i$ and iterate. For an illustration of this case, see Figure \ref{fig:2}.

\begin{figure}
\begin{tikzpicture}[x=0.75pt,y=0.75pt,yscale=-0.8,xscale=0.8]

\draw    (91,155.5) .. controls (92,32.5) and (485,48.5) .. (552,54.5) ;
\draw    (91,155.5) .. controls (93,233.5) and (452,250.5) .. (570,236.5) ;
\draw    (552,54.5) .. controls (606,91.5) and (593,201.5) .. (570,236.5) ;
\draw  [dash pattern={on 4.5pt off 4.5pt}]  (552,54.5) .. controls (542,97.5) and (526,163.5) .. (570,236.5) ;
\draw    (283,132.5) .. controls (303,82.5) and (250,99.17) .. (269,60.17) ;
\draw    (287,230.5) .. controls (327,200.5) and (267,178.5) .. (283,132.5) ;
\draw    (402,111.5) .. controls (449,98.5) and (423,144.5) .. (449,149.5) ;
\draw    (416,191.5) .. controls (464,194.5) and (487,156.5) .. (449,149.5) ;
\draw    (402,111.5) .. controls (365,121.5) and (361,185.5) .. (416,191.5) ;
\draw    (451.6,184.6) -- (457.35,191.83) ;
\draw [shift={(458.6,193.4)}, rotate = 231.5] [color={rgb, 255:red, 0; green, 0; blue, 0 }  ][line width=0.75]    (6.56,-1.97) .. controls (4.17,-0.84) and (1.99,-0.18) .. (0,0) .. controls (1.99,0.18) and (4.17,0.84) .. (6.56,1.97)   ;
\draw    (468.4,163) -- (477.45,161.03) ;
\draw [shift={(479.4,160.6)}, rotate = 167.69] [color={rgb, 255:red, 0; green, 0; blue, 0 }  ][line width=0.75]    (6.56,-1.97) .. controls (4.17,-0.84) and (1.99,-0.18) .. (0,0) .. controls (1.99,0.18) and (4.17,0.84) .. (6.56,1.97)   ;
\draw    (438,141.4) -- (446.97,134.97) ;
\draw [shift={(448.6,133.8)}, rotate = 144.36] [color={rgb, 255:red, 0; green, 0; blue, 0 }  ][line width=0.75]    (6.56,-1.97) .. controls (4.17,-0.84) and (1.99,-0.18) .. (0,0) .. controls (1.99,0.18) and (4.17,0.84) .. (6.56,1.97)   ;
\draw    (428.8,113.8) -- (435.27,106.5) ;
\draw [shift={(436.6,105)}, rotate = 131.55] [color={rgb, 255:red, 0; green, 0; blue, 0 }  ][line width=0.75]    (6.56,-1.97) .. controls (4.17,-0.84) and (1.99,-0.18) .. (0,0) .. controls (1.99,0.18) and (4.17,0.84) .. (6.56,1.97)   ;
\draw    (390.8,117) -- (385.3,108.67) ;
\draw [shift={(384.2,107)}, rotate = 56.58] [color={rgb, 255:red, 0; green, 0; blue, 0 }  ][line width=0.75]    (6.56,-1.97) .. controls (4.17,-0.84) and (1.99,-0.18) .. (0,0) .. controls (1.99,0.18) and (4.17,0.84) .. (6.56,1.97)   ;
\draw    (374.8,150.2) -- (365.4,149.87) ;
\draw [shift={(363.4,149.8)}, rotate = 2.01] [color={rgb, 255:red, 0; green, 0; blue, 0 }  ][line width=0.75]    (6.56,-1.97) .. controls (4.17,-0.84) and (1.99,-0.18) .. (0,0) .. controls (1.99,0.18) and (4.17,0.84) .. (6.56,1.97)   ;
\draw    (405.2,190.2) -- (403.22,196.3) ;
\draw [shift={(402.6,198.2)}, rotate = 288] [color={rgb, 255:red, 0; green, 0; blue, 0 }  ][line width=0.75]    (6.56,-1.97) .. controls (4.17,-0.84) and (1.99,-0.18) .. (0,0) .. controls (1.99,0.18) and (4.17,0.84) .. (6.56,1.97)   ;
\draw    (413.2,98.7) .. controls (345.4,97.8) and (345.8,199.4) .. (421,202.2) ;
\draw    (413.2,98.7) .. controls (464.2,98.6) and (436.2,141.2) .. (462.2,146.2) ;
\draw    (421,202.2) .. controls (475.4,202.2) and (500.2,153.2) .. (462.2,146.2) ;

\draw (601,133.4) node [anchor=north west][inner sep=0.75pt]    {$\partial N_{i}$};
\draw (193,140.4) node [anchor=north west][inner sep=0.75pt]    {$K$};
\draw (405,145.4) node [anchor=north west][inner sep=0.75pt]    {$\Omega $};
\draw (395.2,72.4) node [anchor=north west][inner sep=0.75pt]    {$\partial \tilde{\Omega }$};

\end{tikzpicture}
\caption{$K \cap \Omega = \varnothing$.}
\label{fig:2}
\end{figure}

\subsection{Outline of the proof of Theorem \ref{thm:main_compact}}
It is tempting to try to use the maximum principle of \cite{white_maxprinciple} after the pull tight step of the min-max program to ensure that every element in the critical set of a minimizing sequence is contained in $\interior(N)$. However, these elements are only known to be varifolds with bounds on the first variation of mass, meaning that we do not have any information about how they are \emph{oriented} relative to $\partial N$. Since we only require the one-sided bound on $h$, this argument will not work if $|h| > H_{\partial N}$.

Instead, we carry out all of the steps of the min-max program as if we were doing min-max with an obstacle (as in \cite{zhihan}). After proving the existence of a varifold in the critical set of a pulled tight minimizing sequence that is almost minimizing, we can use the regularity for the obstacle problem (see \cite{tamanini}) to show that this varifold is $C^{1,\alpha}$. This regularity is sufficient to ensure that the hypersurface has the correct orientation at the boundary, at which point the classical Hopf lemma for elliptic second order operators applies to rule out boundary touching. Once we know that this hypersurface is contained in $\interior(N)$, the local regularity results from \cite{zhou-zhu_PMC} apply unchanged.

\subsection{Organization of the paper}
In \S\ref{sec:conventions}, we establish the basic notations and conventions for the rest of the paper.

In Part \ref{part:compact}, we develop a localized min-max theory for PMC hypersurfaces in compact manifolds with good boundaries. In \S\ref{sec:compact_setup}-\ref{sec:generic}, we modify the standard min-max machinery to adapt to the setting with boundary. In \S\ref{sec:localization} we develop a new localization argument essential for the results in noncompact manifolds.

In Part \ref{part:noncompact}, we use the theory from Part \ref{part:compact} to prove a localized min-max theorem for PMC hypersurfaces in complete noncompact manifolds. In \S\ref{sec:localization_noncompact} we prove the main general theorem of the paper. In \S\ref{sec:app1} we apply this theory to the problem of PMC hypersurfaces in Euclidean space. In \S\ref{sec:app2} we apply this theory to the problem of CMC hypersurfaces in finite volume manifolds.

\begin{acknowledgements}
    I am grateful to Fernando Cod{\'a} Marques for an insightful discussion about this work.
\end{acknowledgements}

\section{Conventions and notation}\label{sec:conventions}
Let $(M, g)$ be a complete Riemannian manifold of dimension $n+1$. Let $N\subset M$ be a compact domain with smooth boundary.

\subsection{Basic notations}
We denote
\begin{itemize}
    \item $B_r(x) = \{y \in M \mid \mathrm{dist}_g(x, y) < r\}$,
    \item $B_r^N(x) = B_r(x) \cap N$,
    \item $A_{s,r}(x) = \{y \in M \mid s < \mathrm{dist}_g(x,y) < r\}$,
    \item $A_{s,r}^N(x) = A_{s,r}(x) \cap N$,
    \item $\cH^k$ is the $k$-dimensional Hausdorff measure on $(M, g)$,
    \item $\interior(A)$ for $A \subset M$ is the union of all open sets $W$ in $M$ satisfying $W \subset A$.
\end{itemize}
We say $U \subset N$ is \emph{relatively open} if there is an open set $\tilde{U} \subset M$ with $\tilde{U} \cap N = U$.
\begin{itemize}
    \item $\intrel(A)$ for $A \subset N$ is the union relatively open sets $U$ in $N$ satisfying $U \subset A$.
\end{itemize}

\subsection{Varifolds}
We denote
\begin{itemize}
    \item $\mathcal{V}_k(M)$ is the space of $k$-varifolds in $M$,
    \item $\mathcal{V}_k(N)$ is the space of $k$-varifolds in $M$ with support in $N$,
    \item $\mathbf{F}$ is the metric on $\mathcal{V}_k(M)$ given by
    \[ \mathbf{F}(V_1, V_2) = \sup\{V_1(f) - V_2(f)\mid f \in C^{0,1}_{\text{cpt}}(\mathrm{Gr}_k(M)),\ |f| \leq 1,\ \mathrm{Lip}(f) \leq 1\}. \]
    \item $\|V\|$ is the associated Radon measure on $M$ for $V \in \mathcal{V}_k(M)$,
    \item $|\Sigma|$ is the integral $k$-varifold associated to $\Sigma$ an immersed $k$-dimensional submanifold,
    \item $\|\Sigma\|$ is the associated Radon measure on $M$ for the $k$-varifold $|\Sigma|$, where $\Sigma$ is an immersed $k$-dimensional submanifold,
    \item $\delta V(X)$ is the first variation of the mass of $V \in \mathcal{V}_k(M)$ under the flow generated by the vector field $X$,
    \item $\mathrm{VarTan}(V, x)$ is the set of varifold tangent cones at $x \in M$ for $V$ a $k$-varifold with bounded first variation.
\end{itemize}
Let $\bar{\nu}$ denote the unit normal vector field along $\partial N$ pointing into $N$. Let $U \subset N$ be relatively open. We denote
\begin{itemize}
    \item $\mathcal{X}_{\mathrm{cpt}}(U)$ is the space of smooth vector fields with compact support in $U$,
    \item $\mathcal{X}_{\mathrm{cpt}}^+(U) = \{X \in \mathcal{X}_{\mathrm{cpt}}(U) \mid \langle X, \bar{\nu}\rangle \geq 0\}$.
\end{itemize}
Fix $c > 0$.

\begin{definition}[Bounded constrained first variation]
    $V \in \mathcal{V}_k(N)$ has \emph{$c$-bounded constrained first variation} in $U$ if
    \[ \delta V(X) \geq -c\int_U |X|\ d\|V\| \]
    for all $X \in \mathcal{X}_{\text{cpt}}^+(U)$.
\end{definition}

\begin{definition}[Bounded first variation]
    $V \in \mathcal{V}_k(N)$ has \emph{$c$-bounded first variation} in $U$ if there is an open set $\tilde{U} \subset M$ satisfying $\tilde{U} \cap N = U$ and
    \[ |\delta V(X)| \leq c\|V\|(U)\|X\|_{C^0} \]
    for all $X \in \mathcal{X}_{\text{cpt}}(\tilde{U})$.
\end{definition}

The following lemma is an adaptation of \cite[Lemma 2.6]{zhihan}.

\begin{lemma}\label{lem:constrained_variation}
    Let $V \in \mathcal{V}_k(N)$ have $c$-bounded constrained first variation in $N$. Then there is a constant $C = C(N, c) < \infty$ so that $V$ has $C$-bounded first variation in $N$.
\end{lemma}
\begin{proof}
    Let $r > 0$ small enough so that the normal exponential map along $\partial N$ is a diffeomorphism. Extend $\bar{\nu}$ in $B_r(\partial N)$ by $\nabla_{\bar{\nu}}\bar{\nu} = 0$. Let $\eta \in C_{\text{cpt}}^1(B_r(\partial N), [0, 1])$ so that $\eta \equiv 1$ in a neighborhood of $\partial N$.

    Let $X \in \mathcal{X}_{\text{cpt}}(M)$ satisfy $\|X\|_{C^0} = 1$ (without loss of generality). Set $X_0 = X - \eta\langle X, \bar{\nu}\rangle \bar{\nu}$. By construction, $\pm X_0\mid_N \in \mathcal{X}_{\text{cpt}}^+(U)$, $1 \pm \langle X, \bar{\nu}\rangle \geq 0$, and $|X_0| \leq |X| \leq 1$. We compute
    \begin{align*}
        \delta V(X)
        & = \int (-\mathrm{div}_S(-X_0) + \mathrm{div}_S(\eta\langle X, \bar{\nu}\rangle\bar{\nu}))\ dV(x, S)\\
        & \leq c\|V\|(N) - \int (\mathrm{div}_S(\eta(1-\langle X, \bar{\nu}\rangle)\bar{\nu}) + \mathrm{div}_S(\eta\bar{\nu}))\ dV(x, S)\\
        & \leq c\|V\|(N) + c\|V\|(N)\|\eta(1-\langle X, \bar{\nu}\rangle)\|_{C^0} + \|\eta \bar{\nu}\|_{C^1}\|V\|(N)\\
        & \leq (3c+\|\eta \bar{\nu}\|_{C^1})\|V\|(N),
    \end{align*}
    and
    \begin{align*}
        \delta V(X)
        & = \int (\mathrm{div}_SX_0 + \mathrm{div}_S(\eta\langle X, \bar{\nu}\rangle\bar{\nu}))\ dV(x, S)\\
        & \geq -c\|V\|(N) + \int (\mathrm{div}_S(\eta(1+\langle X, \bar{\nu}\rangle)\bar{\nu}) - \mathrm{div}_S(\eta\bar{\nu}))\ dV(x, S)\\
        & \geq -c\|V\|(N) - c\|V\|(N)\|\eta(1+\langle X, \bar{\nu}\rangle)\|_{C^0} - \|\eta \bar{\nu}\|_{C^1}\|V\|(N)\\
        & \geq -(3c+\|\eta \bar{\nu}\|_{C^1})\|V\|(N).
    \end{align*}
    The conclusion follows by taking $C = 3c + \|\eta \bar{\nu}\|_{C^1}$.
\end{proof}

\subsection{Currents}
We denote
\begin{itemize}
    \item $\mathbf{I}_k(M)$ is the space of $k$-dimensional integral currents in $M$,
    \item $\mathbf{I}_k(N)$ is the space of $k$-dimensional integral currents in $M$ with support in $N$,
    \item $[\Sigma]$ is the $k$-dimensional integral current associated to an oriented immersed $k$-dimensional submanifold of $M$,
    \item $|T|$ is the associated integral $k$-varifold for $T \in \mathbf{I}_k(M)$,
    \item $\|T\|$ is the associated Radon measure on $M$ for $T \in \mathbf{I}_k(M)$.
\end{itemize}
The space of integral currents has various norms and metrics:
\begin{itemize}
    \item $\mathbf{M}$ is the mass norm on $\mathbf{I}_k(M)$,
    \item $\mathcal{F}$ is the flat norm on $\mathbf{I}_k(M)$,
    \item $\mathbf{F}(T, S)$ for $T, S \in \mathbf{I}_k(M)$ is given by
    \[ \mathbf{F}(T, S) = \mathcal{F}(T - S) + \mathbf{F}(|T|, |S|). \]
\end{itemize}

\subsection{Finite perimeter sets}
We denote
\begin{itemize}
    \item $\mathcal{C}(M)$ is the collection of sets of finite perimeter in $M$,
    \item $\partial^*\Omega$ is the reduced boundary for $\Omega \in \mathcal{C}(M)$,
    \item $|\partial^*\Omega|$ is the integral $n$-varifold associated to $\partial^*\Omega$,
    \item $[\partial^*\Omega]$ is the $n$-dimensional integral current associated to $\partial^*\Omega$ for $\Omega \in \mathcal{C}(M)$.
\end{itemize}
$\mathcal{C}(M)$ naturally inherits 3 metrics from $\mathbf{I}_n(M)$:
\begin{itemize}
    \item $\mathbf{M}([\partial^*\Omega])$,
    \item $\mathcal{F}([\partial^*\Omega])$,
    \item $\mathbf{F}([\partial^*\Omega_1], [\partial^*\Omega_2])$.
\end{itemize}
By definition, we have
\begin{equation}\label{eqn:metrics}
    \mathcal{F}([\partial^*\Omega]) \leq \mathbf{F}([\partial^*\Omega], 0)\ \ \text{and}\ \ \mathcal{F}([\partial^*\Omega]) \leq \mathbf{M}([\partial^*\Omega]).
\end{equation}
We recall the following fact when $M$ is closed, which can be found for instance in \cite[Lemma 3.4]{inauen-marchese}.

\begin{lemma}\label{lem:flat_dominates_L1}
    There is a constant $\eps_0 > 0$ (depending on $M$ closed) so that if $\mathcal{F}([\partial^*\Omega]) \leq \eps_0$ and $\cH^{n+1}(\Omega) \leq \frac{1}{2}\cH^{n+1}(M)$, then $\cH^{n+1}(\Omega) = \mathcal{F}([\partial^* \Omega])$.
\end{lemma}

By \eqref{eqn:metrics} and Lemma \ref{lem:flat_dominates_L1}, convergence in any of the above metrics implies $L^1$ convergence of the indicator functions of the finite perimeter sets.

\subsection{Hypersurfaces and mean curvature}
Let $\phi : \Sigma \to (M, g)$ be a two-sided immersed hypersurface, and let $\nu$ be a continuous choice of unit normal vector field along $\Sigma$. We let $H_{\Sigma, \nu}$ denote the mean curvature of $H$ with respect to $\nu$. To specify our convention, for $\Sigma = \partial B_1^{\R^{n+1}}(0)$ and $\nu(x) = x$ (the outward pointing unit normal vector field for $B_1^{\R^{n+1}}(0)$), we have $H_{\Sigma, \nu} = n$. In this convention, for a smooth vector field $X$ on $M$, we have
\[ \delta |\Sigma|(X) = \int_{\Sigma} \langle X, H_{\Sigma, \nu}\nu\rangle d\|\Sigma\|. \]
If $\phi(\Sigma) = \partial \Omega$, we write $H_{\Sigma, \Omega}$ to mean $H_{\Sigma, \nu}$ for $\nu$ the outward pointing unit normal vector field along $\Sigma$.

\begin{definition}[Almost embedding]
    $\Sigma$ is \emph{almost embedded} if for any $p \in \phi(\Sigma)$ where $\phi$ fails to be an embedding, there is a neighborhood $W$ of $p$ in $M$ so that $\Sigma \cap \phi^{-1}(W) = \sqcup_{i=1}^N \Sigma_i$, $\phi\mid_{\Sigma_i}$ is an embedding for all $i$, and for each $i \neq j$, $\phi(\Sigma_i)$ lies on one side of $\phi(\Sigma_j)$ in $W$.
\end{definition}

\begin{definition}[Prescribed/constant mean curvature]
    For $h \in C^{\infty}_{\text{loc}}(M)$, a two-sided immersion is an \emph{$h$-PMC hypersurface} if there is a choice of continuous unit normal vector field along $\Sigma$ satisfying $H_{\Sigma, \nu} = h\mid_{\Sigma}$. If $h \equiv c$ constant, then $\Sigma$ is a \emph{$c$-CMC} hypersurface.
\end{definition}

\subsection{Second variation and index}
If $\Sigma = \partial \Omega$ is an $h$-PMC hypersurface, then $\Sigma$ is a critical point\footnote{Just like for the area functional, we can make sense of being a critical point over compactly supported variations even if the functional itself is undefined for $\Omega$.} for the functional
\[ \cA^h(\Omega') = \cH^{n}(\partial^*\Omega') - \int_{\Omega'} h. \]
Let $\nu$ be the unit normal vector field along $\Sigma$ pointing out of $\Omega$. The second variation of $\cA^h$ for $\Sigma$ for any compactly supported normal variation $\varphi \nu$ is given by (see \cite[(1.3)]{zhou_mult})
\begin{equation}\label{eqn:second_variation}
    Q_{\Sigma}^h(\varphi, \varphi) = \int_\Sigma (|\nabla^{\Sigma} \varphi|^2 - (\mathrm{Ric}_M(\nu, \nu) + |A_{\Sigma}|^2 + \partial_\nu h)\varphi^2),
\end{equation}
where $\mathrm{Ric}_M$ is the Ricci curvature of $(M, g)$ and $A_{\Sigma}$ is the second fundamental form of the immersion $\Sigma$.

\begin{definition}[Stability]
    Let $U \subset M$ open. $\Sigma$ is \emph{$\cA^h$-stable} in $U$ if $Q_{\Sigma}^h$ is positive semi-definite on $C^{\infty}_{\text{cpt}}(U)$ (otherwise $\Sigma$ is \emph{$\cA^h$-unstable}). $\Sigma$ is \emph{strictly $\cA^h$-stable} in $U$ if $Q_{\Sigma}^h$ is positive definite on $C^{\infty}_{\text{cpt}}(U)$.
\end{definition}

\begin{definition}[Index]
    Let $U \subset M$ open. We define
    \[ \mathrm{index}_{\cA^h, U}(\Sigma) = \sup\{\mathrm{dim}(L) \mid L \text{\ is a subspace of\ } C^{\infty}_{\text{cpt}}(U),\ Q_{\Sigma}^h\mid_L \text{is negative definite}\}. \]
    If $U = M$, we omit $U$ from the subscript.
\end{definition}

\begin{definition}[Nondegenerate]
    Suppose $\Sigma$ is closed. $\Sigma$ is \emph{$\cA^h$-nondegenerate} if the operator
    \[ \mathcal{L}_{\Sigma}^h = -\Delta_{\Sigma} - (\mathrm{Ric}_M(\nu, \nu) + |A_{\Sigma}|^2 + \partial_\nu h) \]
    has trivial kernel\footnote{If $\Sigma$ is nondegenerate and stable, then $\Sigma$ is strictly stable.}.
\end{definition}

\part{Compact Manifolds with Barrier}\label{part:compact}

In this part of the paper, we develop the tools for a localized min-max theorem in compact manifolds whose boundary provides a good barrier.

\section{Setup}\label{sec:compact_setup}

Let $(N, \partial N, g)$ be a smooth compact Riemannian manifold with smooth boundary. We always assume that the dimension of $N$ is $3 \leq n+1 \leq 7$.

Let $(\tilde{N}, \tilde{g})$ be a smooth closed Riemannian manifold of dimesion $n+1$ which contains $(N, \partial N, g)$ isometrically.

\subsection{Space of subsets}\label{subsec:subsets}
We fix a disjoint decomposition by connected components $\tilde{N} \setminus N = E_+ \sqcup E_-$ and let $P_{\pm} = \partial E_{\pm}$.

We let $\cCEE(N)$ denote the set of all $\Omega \in \mathcal{C}(\tilde{N})$ satisfying
\[ E_+ \subset \Omega\ \ \text{and}\ \ E_- \cap \Omega = \varnothing. \]
Note that $\cCEE(N)$ naturally inherits any metric or topology on $\mathcal{C}(\tilde{N})$.

\subsection{Functional}\label{subsec:prescription}
Fix $h \in C^{\infty}(N)$. In an abuse of notation, we also let $h$ denote a fixed smooth extension of $h$ to $\tilde{N}$. We let $c := \sup_{\tilde{N}} |h|$. For $\Omega \in \cCEE(N)$, we define
\[ \cA^h(\Omega) = \cH^n(\partial^*\Omega) - \int_{\Omega} h. \]

\subsection{Barrier}
We define what it means for $\partial N$ to be a good barrier for min-max constructions for $\cA^h$ over $\cCEE(N)$.

\begin{definition}
    $\partial N$ is a \emph{barrier} for $(h, P_+, P_-)$ if
    \[ H_{P_+, E_+} < h\mid_{P_+}\ \ \text{and}\ \ H_{P_-, E_-} < -h\mid_{P_-}. \]
\end{definition}

\subsection{Min-max setup}\label{subsec:min-max_setup}
Let $X^k$ be a connected cubical complex of dimension $k$ in some $I^m = [0, 1]^m$. Let $Z \subset X$ be a nonempty cubical subcomplex.

Let $\Phi_0 : X \to \cCEE(N)$ be a continuous map in the $\mathbf{F}$-topology\footnote{We can assume, without loss of generality by \cite[Remarks 15.2 and 15.3]{maggi}, that $\partial\Phi_0(x) = \overline{\partial^*\Phi_0(x)}$ for all $x \in X$.}.

\begin{definition}[Homotopy class]
    Let $\Pi$ be the set of all sequences of $\mathbf{F}$-continuous maps $\{\Phi_i : X \to \cCEE(N)\}_i$ so that there are $\mathcal{F}$-continuous homotopies $\{\Psi_i : [0, 1] \times X \to \cCEE(N)\}_i$ satisfying $\Psi_i(0, \cdot) = \Phi_i$, $\Psi_i(1, \cdot) = \Phi_0$, and
    \[ \limsup_{i\to\infty} \sup\{\mathbf{F}(\Psi_i(t, z), \Phi_0(z)) \mid t \in [0, 1],\ z \in Z\} = 0. \]
    Such a sequence $\{\Phi_i\}$ is called an \emph{$(X, Z)$-homotopy sequence of mappings to $\cCEE(N)$}, and $\Pi$ is called the \emph{$(X, Z)$-homotopy class of $\Phi_0$}.
\end{definition}

\begin{definition}[Width]
    The \emph{$h$-width} of $\Pi$ is
    \[ \mathbf{L}^h(\Pi) = \inf_{\{\Phi_i\} \in \Pi} \limsup_{i \to \infty} \sup_{x \in X} \cA^h(\Phi_i(x)). \]
\end{definition}

\begin{definition}[Minimizing sequence]
    We say $\{\Phi_i\}_i \in \Pi$ is a \emph{minimizing sequence} if
    \[ \limsup_{i \to \infty} \sup_{x \in X} \cA^h(\Phi_i(x)) = \mathbf{L}^h(\Pi). \]
\end{definition}

\begin{definition}[Critical set]
    The \emph{critical set} of a minimizing sequence $\{\Phi_i\}_i \in \Pi$ is
    \[ \mathbf{C}(\{\Phi_i\}) = \left\{V \in \mathcal{V}_n(N) \mid V = \lim_{j\to \infty} |\partial^*\Phi_{i_j}(x_j)|,\ \lim_{j \to \infty} \cA^h(\Phi_{i_j}(x_j)) = \mathbf{L}^h(\Phi)\right\}. \]
\end{definition}

\section{Preliminary tools}
Here we collect the standard facts required for the min-max construction.

\subsection{Local regularity for obstacle problem}

We recall that the dimension of $N$ is assumed to be $3 \leq n+1 \leq 7$.

\begin{theorem}\label{thm:obstacle}
    Let $U \subset N$ be relatively open. Suppose $\Omega \in \cCEE(N)$ minimizes $\cA^h$ in $\{\Omega' \in \mathcal{C}(\tilde{N}) \mid \Omega \triangle \Omega' \subset\subset U\} \subset \cCEE(N)$. Then $\partial^*\Omega \cap U = \partial \Omega \cap U$ is a $C^{1,\alpha}$ embedded hypersurface in $U$.
\end{theorem}
\begin{proof}
    Let $\Omega' \subset \mathcal{C}(\tilde{N})$ satisfy $\Omega \triangle \Omega' \subset B_r(x)\cap N \subset U$. Then
    \begin{align*}
        \cH^n(\partial^*\Omega) - \cH^n(\partial^*\Omega')
        & = \cA^h(\Omega) - \cA^h(\Omega') + \int_{\Omega} h - \int_{\Omega'} h\\
        & \leq \int_{\Omega \setminus \Omega'} h - \int_{\Omega' \setminus \Omega} h\\
        & \leq c\cH^{n+1}(B_r(x))\\
        & \leq c\omega r^{n+1},
    \end{align*}
    where $\omega$ is a constant depending on $(N, g)$ and $r_0$ for all $r < r_0$. In other words, $\Omega$ is almost minimizing with respect to the obstacle $\partial N$. Hence, the arguments of \cite[\S3]{tamanini} apply to give the desired regularity (working only in the assumed dimensions).
\end{proof}

When $\partial N$ is a barrier for $(h, P_+, P_-)$, we use the maximum principle and the above $C^{1, \alpha}$ regularity to show that local $\cA^h$-minimizers lie in $\interior(N)$.

\begin{corollary}\label{cor:local_reg}
    Let $U \subset N$ be relatively open. Suppose $\Omega \in \cCEE(N)$ minimizes $\cA^h$ in $\{\Omega' \in \mathcal{C}(\tilde{N}) \mid \Omega \triangle \Omega' \subset\subset U\} \subset \cCEE(N)$. Further suppose that $\partial N$ is a barrier for $(h, P_+, P_-)$. Then $\partial^*\Omega \cap U = \partial \Omega \cap U$ is a smooth embedded $\cA^h$-stable hypersurface in $U \cap \interior(N)$ with $H_{\partial \Omega, \Omega} = h\mid_{\partial \Omega}$.
\end{corollary}
\begin{proof}
    By the standard interior regularity (see \cite[Theorem 2.2]{zhou-zhu_PMC}), $\partial \Omega \cap U \cap \mathrm{Int}(N)$ is a smooth embedded stable $h$-PMC hypersurface. Suppose for contradiction that a component $\Sigma$ of $\partial \Omega \cap U$ touches $\partial N \cap U$.
    
    We first claim that $\Sigma$ must be identically $\partial N \cap U$. Suppose otherwise, so $\Sigma$ contains a point in $\interior(N) \cap U$. By Theorem \ref{thm:obstacle}, we can find a ball $B \subset \partial N$ and a relatively open set $W \subset U$ so that (in normal Fermi coordinates over $\partial N$) $\Sigma \cap W$ is the graph of a $C^{1,\alpha}$ function $u$ on $\overline{B}$ satisfying $u\mid_{\interior(B)} > 0$, $u(y) = 0$ for some $y \in \partial B$, and $\nabla u(y) = 0$. Since $\Omega \in \cCEE(N)$, we have (depending on whether $\partial N \cap W \subset P_+$ or $\partial N \cap W \subset P_-$) that $\overline{\Omega} \cap W$ is the sublevel graph of $u$ (respectively the suplevel graph of $u$). We now arrive at a contradiction by the interior regularity, the Hopf lemma, and the assumption that $\partial N$ is a barrier. So we conclude that $\Sigma$ is identically $\partial N \cap U$, and the rest of $\partial \Omega \cap U$ lies in $U \cap \mathrm{Int}(N)$ by embeddedness.
    
    Since $\partial N$ is a barrier and $\Omega \in \cCEE(N)$ minimizes $\cA^h$ in $U$, the first variation of $\cA^h$ for any inward pointing vector field with small compact support in $U$ gives a contradiction, so we conclude that $\partial \Omega \cap U \subset \interior(N)$.
\end{proof}

\subsection{Unique continuation}\label{subsec:unique_continuation} Let $\tilde{\mathcal{S}} \subset C^{\infty}(\tilde{N})$ consist of smooth Morse functions $\tilde{h} \in C^{\infty}(\tilde{N})$ so that $\tilde{h}^{-1}(0)$ is a smooth closed hypersurface whose mean curvature vanishes to at most finite order. By \cite[Proposition 0.2]{zhou-zhu_PMC}, $\tilde{\mathcal{S}}$ is an open and dense subset of $C^{\infty}(\tilde{N})$.

Let $\mathcal{S} \subset C^{\infty}(N)$ consist of those $h \in C^{\infty}(N)$ so that there is some $\tilde{h} \in \tilde{\mathcal{S}}$ with $\tilde{h}\mid_N = h$.

\begin{proposition}
    $\mathcal{S}$ is a dense subset of $C^{\infty}(N)$.
\end{proposition}
\begin{proof}
    Take $h \in C^{\infty}(N)$. Let $\tilde{h}$ be any extension of $h$ to $\tilde{N}$. By the density of $\tilde{\mathcal{S}}$, there is a sequence $\{\tilde{h}_i\}_i \subset \tilde{\mathcal{S}}$ converging smoothly to $\tilde{h}$. By definition, $\{\tilde{h}_i\mid_N\} \subset \mathcal{S}$ converges smoothly to $\tilde{h}\mid_N = h$, as desired.
\end{proof}

By \cite[Theorem 3.11]{zhou-zhu_PMC}, $h$-PMC hypersurfaces for functions $h \in \mathcal{S}$ satisfy a unique continuation property required for the proof of regularity of almost minimizers in min-max.

\section{Min-max with barrier}

In this section, we prove the following theorem.

\begin{theorem}\label{thm:basic_min-max}
    Suppose $h \in \mathcal{S}$ and $\partial N$ is a barrier for $(h, P_+, P_-)$. Assume
    \[ \mathbf{L}^h(\Pi) > \sup_{z \in Z} \cA^h(\Phi_0(z)). \]
    Then there is $\Omega \in \cCEE(N)$ so that $\cA^h(\Omega) = \mathbf{L}^h(\Pi)$ and $\Sigma = \partial \Omega$ is a smooth, closed, almost embedded hypersurface contained in $\mathrm{Int}(N)$ satisfying $H_{\Sigma, \Omega} = h\mid_{\Sigma}$.
\end{theorem}

\subsection{Nontriviality}
Here we show that $\partial^* \Omega \neq \varnothing$. If $\partial^* \Omega = \varnothing$, then $\Omega \in \{\varnothing, \tilde{N}\}$\footnote{This scenario is only possible if at least one of $E_-$ or $E_+$ is empty.}. By the isoperimetric inequality in $\tilde{N}$ (see \cite[Theorem 2.3]{zhou-zhu_PMC}), $\varnothing$ and $\tilde{N}$ are strict minimizers for $\cA^h$ in a small flat neighborhood, so any nontrivial min-max $h$-width cannot equal $\cA^h(\varnothing)$ nor $\cA^h(\tilde{N})$ (as $X$ is connected and $Z$ is nonempty).

\subsection{Pull-tight}

Let $c = \sup_N |h|$ and $L = 2(\mathbf{L}^h + c\mathrm{Vol}(\tilde{N}))$. Let
\[ A^L = \{V \in \mathcal{V}_n(N) \mid \|V\|(N) \leq L\} \]
and
\begin{align*}
    A_{\infty}
    & = \{V \in A^L \mid V\ \text{has}\ c\text{-bounded constrained first variation in}\ N\} \cup \{|\partial^* \Phi_0(z)| \mid z \in Z\}.
\end{align*}

We follow the construction of \cite[\S1.2]{zhou_mult}, \cite[\S5]{zhou-zhu_PMC}, and \cite[\S4]{zhou-zhu_CMC} for a map to the space of vector fields $\mathcal{D} : A^L \to \mathcal{X}_{\text{cpt}}(N)$. By the definition of $A_{\infty}$, if $V \in A^L$ satisfies $2^{-j} \leq \mathbf{F}(V, A_{\infty}) \leq 2^{-j+1}$, then we can find $X_V \in \mathcal{X}_{\text{cpt}}^+(N)$ so that
\[ \|X_V\|_{C^1} \leq 1,\ \ \delta V(X_V) + c \int_N |X_V|\ d\|V\| \leq -c_j < 0, \]
where $c_j > 0$ does not depend on $V$. Since the map to vector fields in \cite[\S4.2]{zhou-zhu_CMC} is constructed using linear combinations of the vector fields $X_V$ for some $V$, the identical construction actually gives $\mathcal{D} : A^L \to \mathcal{X}_{\text{cpt}}^+(N)$. Hence, we can use $\mathcal{D}$ to construct the desired tightening map using flows as in \cite[\S4.3, \S4.4]{zhou-zhu_CMC} (the flows stay in $N$ by construction).

To summarize, we obtain an $\mathbf{F}$-continuous map
\[ \mathcal{PT}: [0, 1] \times \{\Omega \in \cCEE(N)\mid |\partial^* \Omega| \in A^L\} \to \cCEE(N) \]
satisfying
\begin{itemize}
    \item $\mathcal{PT}(0, \Omega) = \Omega$ for all $\Omega$,
    \item $\mathcal{PT}(t, \Omega) = \Omega$ for all $t \in [0, 1]$ if $|\partial^* \Omega| \in A_{\infty}$,
    \item there is a continuous function $L : \R_{\geq 0} \to \R_{\geq0}$ with $L(0) = 0$ and $L(t) > 0$ for $t > 0$ so that
    \[ \cA^h(\mathcal{PT}(1, \Omega)) - \cA^h(\Omega) \leq -L(\mathbf{F}(|\partial^* \Omega|, A_{\infty})). \]
\end{itemize}

\begin{lemma}[Pull-tight]
    Let $\{\Phi^*_i\}_i \subset \Pi$ be a minimizing sequence. Then there is another minimizing sequence $\{\Phi_i\}_i \subset \Pi$ so that $\mathbf{C}(\{\Phi_i\}) \subset \mathbf{C}(\{\Phi_i^*\}) \cap A_{\infty}$.
\end{lemma}
\begin{proof}
    The proof of \cite[Lemma 1.8]{zhou_mult} applies.
\end{proof}

\subsection{Existence of almost minimizers}

We make the same definition of $h$-almost minimizing as in \cite[Definition 6.1]{zhou-zhu_PMC}.

\begin{definition}
    For $\eps, \delta > 0$ and $U \subset N$ relatively open, we define $\mathscr{A}^h(U; \eps, \delta)$ to be the set of $\Omega \in \cCEE(N)$ such that if $\Omega = \Omega_0,\ \Omega_1, \hdots, \Omega_m \in \cCEE(N)$ satisfies
    \begin{itemize}
        \item $\Omega_i \triangle \Omega \subset U$,
        \item $\mathcal{F}([\partial^*\Omega_i] - [\partial^*\Omega_{i+1}]) \leq \delta$,
        \item $\cA^h(\Omega_i) \leq \cA^h(\Omega) + \delta$,
    \end{itemize}
    then $\cA^h(\Omega_m) \geq \cA^h(\Omega) - \eps$.

    We say that $V \in \mathcal{V}_n(N)$ is \emph{$h$-almost minimizing in $U$} if there are $\eps_i, \delta_i \to 0$ and $\Omega_i \in \mathscr{A}^h(U; \eps_i, \delta_i)$ such that $|\partial^*\Omega_i| \rightharpoonup V$.
\end{definition}

\begin{proposition}
    If $V \in \mathcal{V}_n(N)$ is $h$-almost minimizing in $U$, then $V$ has $c$-bounded constrained first variation in $U$ for $c = \sup|h|$.
\end{proposition}
\begin{proof}
    The same proof as \cite[Lemma 5.2]{zhou-zhu_CMC} applies.
\end{proof}

\begin{definition}
    We say $V \in \mathcal{V}_n(N)$ is \emph{$h$-almost minimizing in small annuli} if for any $p \in N$, there exists $r_{\text{am}}(p) > 0$ such that $V$ is $h$-almost minimizing in $A_{s,r}^N(p)$ for all $0 < s < r \leq r_{\text{am}}(p)$.
\end{definition}

\begin{theorem}
    For every pulled-tight minimizing sequence $\{\Phi_i\}$ in $\Pi$, there is some $V \in \mathbf{C}(\{\Phi_i\})$ that has $c$-bounded constrained first variation in $N$ and is $h$-almost minimizing in small annuli.
\end{theorem}
\begin{proof}
    The proof follows from the combinatorial construction of \cite[4.9, 4.10]{pitts} (see \cite[Theorem 1.16]{zhou_mult}) and the discretization and interpolation constructions of \cite[\S1.3]{zhou_mult}. That $V$ has $c$-bounded constrained first variation in $N$ follows from the same argument as the final paragraph of the proof of \cite[Theorem 1.7]{zhou_mult}.
\end{proof}

\subsection{Regularity of almost minimizers}

The interior regularity theory of \cite[Theorem 7.1]{zhou-zhu_PMC} is local, and therefore applies to $\mathrm{Int}(N)$. Hence, we only need to deal with the boundary. Throughout this subsection, we exploit the fact that $c$-bounded contrained first variation in $N$ implies $C$-bounded first variation in $\tilde{N}$ for some $C(N, c)$ (by Lemma \ref{lem:constrained_variation}).

We first recall the construction of replacements for $h$-almost minimizing varifolds.

\begin{lemma}[Constrained minimization]\label{lem:constrained_min}
    Let $\eps, \delta > 0$ and $U \subset N$ relatively open. Take $\Omega \in \mathscr{A}^h(U;\eps, \delta)$. Fix a relatively compact subset $K \subset U$. Let $\mathcal{C}_{\Omega}$ be the set of $\Lambda \in \cCEE(N)$ such that there is a sequence $\Omega = \Omega_0,\ \Omega_1, \hdots, \Omega_m=\Lambda$ in $\cCEE(N)$ satisfying
    \begin{itemize}
        \item $\Omega_i \triangle \Omega \subset K$,
        \item $\mathcal{F}([\partial^* \Omega_i] - [\partial^* \Omega_{i+1}]) \leq \delta$,
        \item $\cA^h(\Omega_i)\leq \cA^h(\Omega) + \delta$.
    \end{itemize}
    Then there exists $\Omega^*\in \cCEE(N)$ such that
    \begin{enumerate}
        \item $\Omega^*\in \mathcal{C}_{\Omega}$ and
        \[ \cA^h(\Omega^*) = \inf\{\cA^h(\Lambda) \mid \Lambda \in \mathcal{C}_{\Omega}\}, \]
        \item $\Omega^*$ is locally minimizing for $\cA^h$ in $\intrel(K)$,
        \item $\Omega^*\in \mathscr{A}^h(U; \eps, \delta)$.
    \end{enumerate}
\end{lemma}
\begin{proof}
    The proof follows from the proof of \cite[Lemma 6.7]{zhou-zhu_PMC}, \cite[Lemma 5.7]{zhou-zhu_CMC}.
\end{proof}

\begin{remark}
    If $\partial N$ is a barrier for $(h, P_+, P_-)$, Lemma \ref{lem:constrained_min} (1) and (2) combined with Corollary \ref{cor:local_reg} imply that $\partial^*\Omega^* \cap \intrel(K) = \partial \Omega^* \cap \intrel(K)$ is an $\cA^h$-stable smooth embedded $h$-PMC hypersurface contained in $\interior(N)$.
\end{remark}

\begin{proposition}[Existence of replacements]\label{prop:exist_replacement}
    Let $V \in \mathcal{V}_n(N)$ be $h$-almost minimizing in a relatively open set $U \subset N$. Let $K \subset U$ be a compact subset. Then there exists $V^* \in \mathcal{V}_n(N)$, called an \emph{$h$-replacement} for $V$ in $K$, so that
    \begin{enumerate}
        \item $V \llcorner (N \setminus K) = V^* \llcorner (N \setminus K)$,
        \item $-c\cH^{n+1}(K) \leq \|V\|(N) - \|V^*\|(N) \leq c\cH^{n+1}(K)$,
        \item $V^*$ is $h$-almost minimizing in $U$,
        \item there are $\Omega_i^* \in \mathscr{A}^h(U; \eps_i, \delta_i)$ with $\eps_i, \delta_i \to 0$ so that $\Omega_i^*$ locally minimizes $\cA^h$ in $\intrel(K)$ and $|\partial^*\Omega_i^*| \rightharpoonup V^*$,
        \item if $V$ has $c$-bounded constrained first variation in $N$, then so does $V^*$.
    \end{enumerate}
\end{proposition}
\begin{proof}
    The proof follows \cite[Proposition 6.8]{zhou-zhu_PMC}, \cite[Proposition 5.8]{zhou-zhu_CMC}.
\end{proof}

\begin{proposition}[Regularity of replacements]\label{prop:reg_replace}
    Under the same hypotheses as Proposition \ref{prop:exist_replacement}, if $\partial N$ is a barrier for $(h, P_+, P_-)$, then $V^* \llcorner \intrel(K) = |\Sigma|$ for $\Sigma$ a smooth almost embedded stable $h$-PMC hypersurface in $\intrel(K) \cap \interior(N)$.
\end{proposition}
\begin{proof}
    As in \cite[Proposition 6.9]{zhou-zhu_PMC}, \cite[Proposition 5.9]{zhou-zhu_CMC}, the result follows from the regularity of local minimizers in Corollary \ref{cor:local_reg} (which applies because we assumed that $\partial N$ is a barrier) and the compactness theory for stable $h$-PMC hypersurfaces with uniformly bounded area (see \cite[Theorem 3.6]{zhou-zhu_PMC}). The fact that the limit lies in $\interior(N)$ follows from the fact that $\partial N$ is a barrier and the maximum principle.
\end{proof}

\begin{proposition}[Tangent cones are planar]\label{prop:tangent_plane}
    Suppose $\partial N$ is a barrier for $(h, P_+, P_-)$. Suppose $V \in \mathcal{V}_n(N)$ has $c$-bounded constrained first variation in $N$ and is $h$-almost minimizing in a relatively open set $U \subset N$. For any $y \in \partial N \cap U \cap \mathrm{spt}\|V\|$ and $C \in \mathrm{VarTan}(V, y)$,
    \[ C = m |T_y\partial N| \]
    for some $m \in \N$.
\end{proposition}
\begin{proof}
    The arguments of \cite[Proposition 6.10, 6.11]{zhou-zhu_PMC}, \cite[Proposition 5.10, 5.11]{zhou-zhu_CMC} apply to guarantee that any tangent cone is an integer multiple of a hyperplane in $T_y\tilde{N}$. Since $\mathrm{spt}\|V\| \subset N$, this hyperplane must be contained in the closed halfspace $\overline{T_y^+N}$ and containing the origin, and therefore must equal $T_y\partial N$.
\end{proof}

We can now use the properties of replacements to exclude boundary touching when we choose $h \in \mathcal{S}$.

\begin{proposition}[Interior for almost minimizing]\label{prop:int}
    Suppose $\partial N$ is a barrier for $(h, P_+, P_-)$ and $h \in \mathcal{S}$. Suppose $V \in \mathcal{V}_n(N)$ has $c$-bounded constrained first variation in $N$ and is $h$-almost minimizing in a relatively open set $U \subset N$. Then $\mathrm{spt}\|V\| \cap U \cap \partial N = \varnothing$.
\end{proposition}
\begin{proof}
    Suppose for contradiction that $y \in U \cap \partial N \cap \mathrm{spt}\|V\|$. Let $x \in U \cap \partial N$ with $d(x, y) = d$ small so that $\partial B_d(x)$ intersects $\partial N$ transversally and $B_{2d}(x) \cap N \subset U$.

    Let $V'$ be an $h$-replacement for $V$ in $\overline{B_d^N(x)}$. By Proposition \ref{prop:tangent_plane} and the assumption of transversality, $y \in \mathrm{spt}\|V'\|$. Since $V'$ is almost minimizing, Proposition \ref{prop:tangent_plane} implies that every tangent cone to $V'$ at $y$ is an integer multiple of $T_y\partial N$. Hence, any varifold $V^*$ that agrees with $V'$ on $A_{d/2, d}^N(x)$ has $y \in \mathrm{spt}\|V^*\|$.

    Let $V''$ be an $h$-replacement for $V'$ in $\overline{A_{d/2, 3d/2}^N(x)}$. By Proposition \ref{prop:reg_replace},
    \[ \mathrm{spt}\|V''\| \cap A_{d/2, 3d/2}^N(x) \cap \partial N = \varnothing, \]
    so $y \notin \mathrm{spt}\|V''\|$. However, by the gluing and unique continuation arguments from the proof of \cite[Theorem 7.1]{zhou-zhu_PMC} (which apply because they are local and the gluing happens in $\interior(N)$, and we assumed $h \in \mathcal{S}$), $V''$ agrees with $V'$ on $A_{d/2, d}^N(x)$, so $y \in \mathrm{spt}\|V''\|$ by the above paragraph. Hence, we reach a contradiction.
\end{proof}

We now use a removable singularities result to deduce the same conclusion for varifolds that are $h$-almost minimizing in annuli.

\begin{theorem}
    Suppose $\partial N$ is a barrier for $(h, P_+, P_-)$ and $h \in \mathcal{S}$. Suppose $V \in \mathcal{V}_n(N)$ has $c$-bounded constrained first variation in $N$ and is $h$-almost minimizing in small annuli. Then $V$ is the varifold of a smooth almost embedded $h$-PMC hypersurface in $\mathrm{Int}(N)$.
\end{theorem}
\begin{proof}
    By Proposition \ref{prop:int}, $\mathrm{spt}\|V\|\cap \partial N$ consists of finitely many isolated points. By removable singularities (see \cite[Proof of Theorem 7.1, Step 4]{zhou-zhu_PMC}), $V$ is the varifold of a smooth almost embedded $h$-PMC hypersurface, with finitely many isolated points touching $\partial N$ tangentially. Since $\partial N$ is a barrier, the standard maximum principle implies $\mathrm{spt}\|V\| \cap \partial N = \varnothing$.
\end{proof}

Now that we know that $V$ avoids the boundary, Theorem \ref{thm:basic_min-max} follows as in \cite{zhou-zhu_PMC}.

\section{Min-max with barrier in generic metrics}\label{sec:generic}

We use special results applicable to a generic set of metrics and prescribing functions to upgrade the conclusions of Theorem \ref{thm:basic_min-max}.

\subsection{Good pairs}

Let $(N, \partial N)$ be a compact manifold with smooth boundary. We say a pair $(g, h)$, where $g$ is a smooth Riemannian metric on $N$ and $h \in C^{\infty}(N)$, is a \emph{good pair}, if
\begin{itemize}
    \item $h \in \mathcal{S}(g)$\footnote{Here we include explicitly in the notation the dependence of $\mathcal{S}$ on the metric, since we consider varying metrics on $N$.} as defined in 
    \S\ref{subsec:unique_continuation}
    \item every closed two-sided almost embedded $h$-PMC hypersurface in $\interior(N)$ is embedded and $\cA^h$-nondegenerate.
\end{itemize}

The following result is a consequence of \cite[Lemma 3.5]{zhou_mult}, building on \cite[Proposition 0.2]{zhou-zhu_PMC} and \cite[Theorems 34 and 35]{white_transverse}.

\begin{theorem}\label{thm:generic}
    Let $N$ be a compact manifold with smooth boundary. The set of good pairs is $C^{\infty}$ dense.
\end{theorem}
\begin{proof}
    Let $(g, h)$ be a Riemannian metric and smooth function on $N$. Let $(\tilde{g}, \tilde{h})$ be any extension of $g$ and $h$ to $\tilde{N}$. By \cite[Lemma 3.5]{zhou_mult}, there is a sequence of good pairs $(\tilde{g}_i, \tilde{h}_i)$ for $\tilde{N}$ in the sense of \cite[\S3.3]{zhou_mult} converging smoothly to $(\tilde{g}, \tilde{h})$. Since the restriction to $N$ of a good pair on $\tilde{N}$ is a good pair on $N$ by definition, $(\tilde{g}_i\mid_N, \tilde{h}_i\mid_N)$ is a sequence of good pairs converging smoothly to $(g, h)$, as desired.
\end{proof}

\subsection{Minimality for strict stability} 
We show, based on the arguments of \cite{inauen-marchese}, that a strictly $\cA^h$-stable hypersurface minimizes $\cA^h$ in a small flat neighborhood.

\begin{theorem}\label{thm:minimality}
    Let $(M, g)$ be a smooth closed Riemannian manifold. Let $\Sigma = \partial \Omega$ be a smooth closed embedded $h$-PMC hypersurface which is strictly $\cA^h$-stable. Then there exists $\eps > 0$ (depending on $\Omega$ and $M$) so that
    \[ \cA^h(\Omega') > \cA^h(\Omega) \]
    for all $\Omega' \in \mathcal{C}(M)$ with $0 < \mathcal{F}([\partial \Omega] - [\partial^* \Omega']) < \eps$.
\end{theorem}
\begin{proof}
    We follow the proof of \cite[Proposition 3.1]{inauen-marchese}, but we take extra care because $\cA^h$ is not necessarily an elliptic parametric functional in the sense of \cite[\S2.2]{inauen-marchese} on all of $M$.

    As observed on \cite[Page 205]{white_minimax}, $\cA^h$ on a small tubular neighborhood of $\Sigma$ is an elliptic parametric functional in the sense of \cite[\S1]{white_minimax} and \cite[\S2.2]{inauen-marchese}, so \cite[Theorem 2]{white_minimax} applies.

    We observe that upper bounds on $\cA^h$ imply upper bounds on area (since $h$ is smooth and $M$ is compact), so the arguments of \cite{inauen-marchese} requiring the compactness theory for bounded area apply. Moreover, since flat convergence implies $L^1$ convergence of the indicator functions of the finite perimeter sets (as noted in \S\ref{subsec:subsets}), $\cA^h$ is lower semi-continuous with respect to $\mathcal{F}$-convergence.

    The arguments of \cite[\S3.2]{inauen-marchese} apply verbatim, with the additional observation that \cite[Lemma 3.5]{inauen-marchese} for $\cA^h$ follows directly from \cite[Lemma 3.5]{inauen-marchese} for area combined with \cite[Lemma 3.4]{inauen-marchese} (to handle the integral of $h$ terms).

    The arguments of \cite[\S3.3]{inauen-marchese} preceding \cite[Lemma 3.8]{inauen-marchese} apply to prove \cite[(3.7)]{inauen-marchese} with $F = \cA^h$. Moving the integral of $h$ terms to the right hand side (where we write $\Lambda$ to be the integral $(n+1)$-current supported in $B_r(x_0)$ with $\partial \Lambda = X$), we get
    \[ \mathbf{M}(R_{\delta, \lambda}) \leq \mathbf{M}(R_{\delta, \lambda}+X) + \lambda\mathcal{F}(X) + c\mathbf{M}(\Lambda). \]
    We can now apply \cite[Lemma 3.4]{inauen-marchese} (since we can assume $\mathcal{F}(X) \leq \eps_0$ for the almost minimizing condition) to deduce
    \[ \mathbf{M}(R_{\delta, \lambda}) \leq \mathbf{M}(R_{\delta, \lambda}+X) + (\lambda + c)\mathcal{F}(X). \]
    We can now apply \cite[Lemma 3.8]{inauen-marchese} to conclude that $R_{\delta, \lambda}$ is almost minimizing for mass. Hence, we can apply \cite[Lemma 3.10]{inauen-marchese} and conclude.
\end{proof}

\subsection{Generic min-max theorem}
The following result upgrades the conclusions of Theorem \ref{thm:basic_min-max} for good pairs.

\begin{theorem}\label{thm:generic_min-max}
    Suppose $(g, h)$ is a good pair and $\partial N$ is a barrier for $(h, P_+, P_-)$. Assume
    \[ \mathbf{L}^h(\Pi) > \sup_{z \in Z} \cA^h(\Phi_0(z)). \]
    
    Then there is $\Omega \in \cCEE(N)$ so that $\cA^h(\Omega) = \mathbf{L}^h(\Pi)$ and $\Sigma = \partial \Omega$ is a smooth, closed, embedded hypersurface contained in $\mathrm{Int}(N)$ satisfying $H_{\Sigma, \Omega} = h\mid_{\Sigma}$ and $1 \leq \mathrm{index}_{\cA^h}(\Sigma) \leq k$.
\end{theorem}
\begin{proof}
    The only new conclusions beyond Theorem \ref{thm:basic_min-max} are embeddedness and the index bounds.

    Embeddedness follows immediately from the assumption the $(g, h)$ is a good pair.

    Since $\partial \Omega$ is embedded and contained in $\interior(N)$, the index upper bound of \cite[Theorem 3.6]{zhou_mult} (which is a local deformation argument) applies.

    The index lower bound follows from Theorem \ref{thm:minimality}. Indeed, if $\Sigma$ was $\cA^h$-stable, then $\Sigma$ would be strictly $\cA^h$-stable since $(g, h)$ is a good pair. Then by Theorem \ref{thm:minimality}, $\Omega$ strictly minimizes $\cA^h$ in a flat neighborhood, so any nontrivial min-max $h$-width cannot equal $\cA^h(\Omega)$ (as $X$ is connected and $Z$ is nonempty).
\end{proof}

\subsection{Corollaries of generic min-max} We record two consequence of Theorem \ref{thm:generic_min-max} combined with the compactness theory for $h$-PMC hypersurfaces with uniformly bounded area and index (see \cite[Theorem 2.8]{zhou_mult}).

The first corollary is that the index upper bound applies for metrics $g$ that are not bumpy.

\begin{corollary}
    Suppose $h \in \mathcal{S}$ and $\partial N$ is a barrier for $(h, P_+, P_-)$. Assume
    \[ \mathbf{L}^h(\Pi) > \sup_{z \in Z} \cA^h(\Phi_0(z)). \]
    
    Then there is $\Omega \in \cCEE(N)$ so that $\cA^h(\Omega) = \mathbf{L}^h(\Pi)$ and $\Sigma = \partial \Omega$ is a smooth, closed, almost embedded hypersurface contained in $\mathrm{Int}(N)$ satisfying $H_{\Sigma, \Omega} = h\mid_{\Sigma}$ and $\mathrm{index}_{\cA^h}(\Sigma) \leq k$.
\end{corollary}

The second corollary is that existence (and the index upper bound) applies for $g$ that are not bumpy and $h$ that are not in $\mathcal{S}$.

\begin{corollary}\label{cor:basic_min-max}
    Suppose $\partial N$ is a barrier for $(h, P_+, P_-)$. Assume
    \[ \mathbf{L}^h(\Pi) > \sup_{z \in Z} \cA^h(\Phi_0(z)). \]
    
    Then there is a smooth, closed, two-sided, almost embedded $h$-PMC hypersurface $\Sigma$ contained in $\mathrm{Int}(N)$ satisfying $\cH^n(\Sigma) \leq \mathbf{L}^h(\Pi) + \int_{\{h > 0\}} h$ and $\mathrm{index}_{\cA^h}(\Sigma) \leq k$.
\end{corollary}

\section{Localization in compact manifolds with barrier}\label{sec:localization}

Using our generic min-max result from Theorem \ref{thm:generic_min-max}, we develop a cutting argument to localize min-max hypersurfaces.

\subsection{Setup}\label{subsec:local_setup}
We make a few specialized definitions related to the min-max setup considered above. See Figure \ref{fig:span} for an illustration.

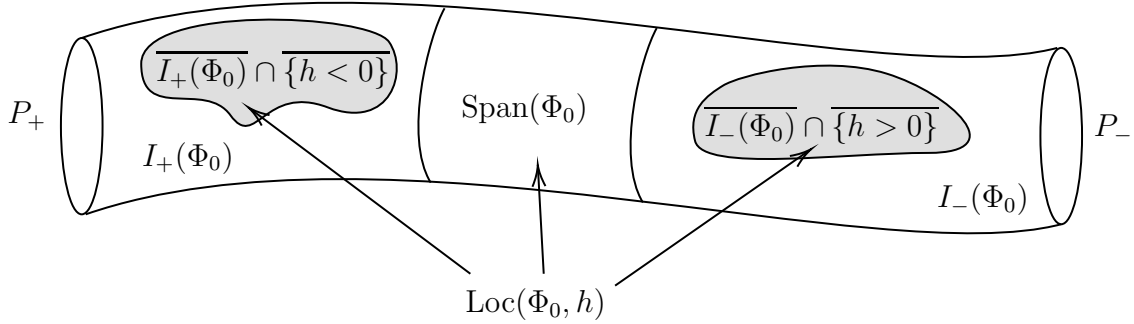
\begin{figure}
\begin{tikzpicture}[x=0.75pt,y=0.75pt,yscale=-1,xscale=1]

\draw    (105,31.5) .. controls (256,-20.5) and (505,60.5) .. (597,36.5) ;
\draw    (107,120.5) .. controls (258,68.5) and (507,149.5) .. (599,125.5) ;
\draw   (94.5,76) .. controls (94.5,51.42) and (99.2,31.5) .. (105,31.5) .. controls (110.8,31.5) and (115.5,51.42) .. (115.5,76) .. controls (115.5,100.58) and (110.8,120.5) .. (105,120.5) .. controls (99.2,120.5) and (94.5,100.58) .. (94.5,76) -- cycle ;
\draw   (588.5,81) .. controls (588.5,56.42) and (593.2,36.5) .. (599,36.5) .. controls (604.8,36.5) and (609.5,56.42) .. (609.5,81) .. controls (609.5,105.58) and (604.8,125.5) .. (599,125.5) .. controls (593.2,125.5) and (588.5,105.58) .. (588.5,81) -- cycle ;
\draw    (288,16.5) .. controls (280,30.5) and (267,81.5) .. (281,104.5) ;
\draw    (395,26.5) .. controls (387,40.5) and (374,91.5) .. (388,114.5) ;
\draw  [fill=gray!25!white ] (135,44.5) .. controls (135,14.5) and (250,19) .. (260,33.5) .. controls (270,48) and (261,78.5) .. (229,66.5) .. controls (197,54.5) and (193,88.5) .. (180,70.5) .. controls (167,52.5) and (135,74.5) .. (135,44.5) -- cycle ;
\draw  [fill=gray!25!white ] (428.47,55.5) .. controls (436.54,50.5) and (478.71,39.5) .. (511,49.5) .. controls (543.29,59.5) and (575,89.5) .. (532,89.5) .. controls (489,89.5) and (418.37,99.5) .. (414.34,83.5) .. controls (410.3,67.5) and (420.39,60.5) .. (428.47,55.5) -- cycle ;
\draw    (301,150.5) -- (190.6,67.7) ;
\draw [shift={(189,66.5)}, rotate = 36.87] [color={rgb, 255:red, 0; green, 0; blue, 0 }  ][line width=0.75]    (10.93,-3.29) .. controls (6.95,-1.4) and (3.31,-0.3) .. (0,0) .. controls (3.31,0.3) and (6.95,1.4) .. (10.93,3.29)   ;
\draw    (338,150.5) -- (335.11,97.5) ;
\draw [shift={(335,95.5)}, rotate = 86.88] [color={rgb, 255:red, 0; green, 0; blue, 0 }  ][line width=0.75]    (10.93,-3.29) .. controls (6.95,-1.4) and (3.31,-0.3) .. (0,0) .. controls (3.31,0.3) and (6.95,1.4) .. (10.93,3.29)   ;
\draw    (374,156.5) -- (471.36,88.64) ;
\draw [shift={(473,87.5)}, rotate = 145.12] [color={rgb, 255:red, 0; green, 0; blue, 0 }  ][line width=0.75]    (10.93,-3.29) .. controls (6.95,-1.4) and (3.31,-0.3) .. (0,0) .. controls (3.31,0.3) and (6.95,1.4) .. (10.93,3.29)   ;

\draw (66,62) node [anchor=north west][inner sep=0.75pt]   [align=left] {$\displaystyle P_{+}$};
\draw (614,68.4) node [anchor=north west][inner sep=0.75pt]    {$P_{-}$};
\draw (295,58) node [anchor=north west][inner sep=0.75pt]   [align=left] {$\displaystyle \mathrm{Span}( \Phi _{0})$};
\draw (134,82.4) node [anchor=north west][inner sep=0.75pt]    {$I_{+}( \Phi _{0})$};
\draw (535,103.4) node [anchor=north west][inner sep=0.75pt]    {$I_{-}( \Phi _{0})$};
\draw (141,36.4) node [anchor=north west][inner sep=0.75pt]    {$\overline{I_{+}( \Phi _{0})} \cap \overline{\{h< 0\}}$};
\draw (417.65,64.4) node [anchor=north west][inner sep=0.75pt]    {$\overline{I_{-}( \Phi _{0})} \cap \overline{\{h >0\}}$};
\draw (297,156.4) node [anchor=north west][inner sep=0.75pt]    {$\mathrm{Loc}( \Phi _{0} ,h)$};

\end{tikzpicture}
\caption{Span, unspanned halves, and localizer.}
\label{fig:span}
\end{figure}

\begin{definition}[Span]
    The \emph{span} of $\Phi_0$ is 
    \[ \Span(\Phi_0) = \overline{\bigcup_{x \in X} \partial^*\Phi_0(x)}. \]
\end{definition}

\begin{definition}[Unspanned halves]
    For any $x_0 \in X$\footnote{Recall that we assume (without loss of generality) that $\partial \Phi_0(x) = \overline{\partial^*\Phi_0(x)}$ for all $x \in X$.}, we define the positive and negative \emph{unspanned halves} by
    \[ I_+(\Phi_0) = (N \cap \interior(\Phi_0(x_0))) \setminus \Span(\Phi_0)\ \ \text{and}\ \ I_-(\Phi_0) = (N \setminus \overline{\Phi_0(x_0)}) \setminus \Span(\Phi_0). \]
    Note that $I_{\pm}(\Phi_0)$ are well-defined: for any $x \in X$ we have
    \[ I_+(\Phi_0) \subset \interior(\Phi_0(x))\ \ \text{and}\ \ I_-(\Phi_0) \cap \overline{\Phi_0(x)} = \varnothing. \]
\end{definition}

\begin{definition}[Localizer]
    The \emph{localizer} of $\Phi_0$ and $h$ is
    \[ \Loc(\Phi_0, h) = \Span(\Phi_0) \cup (\overline{I_+(\Phi_0)} \cap \overline{\{h < 0\}}) \cup (\overline{I_-(\Phi_0)} \cap \overline{\{h > 0\}}). \]
\end{definition}

\subsection{Subdomains}\label{subsec:subdomains}
Here we discuss how to restrict the min-max setup to special subdomains of the original compact manifold with boundary.

Let $\Omega \in \cCEE(N)$ be an open set with smooth boundary in $\interior(N)$ satisfying $\partial \Omega \cap \Span(\Phi_0) = \varnothing$.

\begin{figure}

\begin{tikzpicture}[x=0.75pt,y=0.75pt,yscale=-1,xscale=1]

\fill  [fill=gray!25!white ] (69,54) -- (461.86,54) -- (461.86,116.5) -- (69,116.5) -- cycle ;
\fill  [fill=white ] (229.21,91.5) .. controls (229.21,82.66) and (236.37,75.5) .. (245.21,75.5) -- (248.15,75.5) .. controls (256.99,75.5) and (264.15,82.66) .. (264.15,91.5) -- (264.15,91.5) .. controls (264.15,100.34) and (256.99,107.5) .. (248.15,107.5) -- (245.21,107.5) .. controls (236.37,107.5) and (229.21,100.34) .. (229.21,91.5) -- cycle ;
\draw   (152.51,54) -- (552.19,54) -- (552.19,116.5) -- (152.51,116.5) -- cycle ;
\draw   (315.28,54) -- (389.42,54) -- (389.42,116.5) -- (315.28,116.5) -- cycle ;
\draw    (71.56,36.5) -- (459.3,36.5) ;
\draw    (71.56,36.5) -- (71.73,46.4) ;
\draw    (459.3,36.5) -- (459.3,46.8) ;
\draw    (204.5,145.5) -- (239.93,99.09) ;
\draw [shift={(241.14,97.5)}, rotate = 127.36] [color={rgb, 255:red, 0; green, 0; blue, 0 }  ][line width=0.75]    (10.93,-3.29) .. controls (6.95,-1.4) and (3.31,-0.3) .. (0,0) .. controls (3.31,0.3) and (6.95,1.4) .. (10.93,3.29)   ;
\draw    (180.64,146.5) -- (120,98.74) ;
\draw [shift={(118.43,97.5)}, rotate = 38.23] [color={rgb, 255:red, 0; green, 0; blue, 0 }  ][line width=0.75]    (10.93,-3.29) .. controls (6.95,-1.4) and (3.31,-0.3) .. (0,0) .. controls (3.31,0.3) and (6.95,1.4) .. (10.93,3.29)   ;
\draw    (522.36,144) -- (499.48,103.24) ;
\draw [shift={(498.5,101.5)}, rotate = 60.69] [color={rgb, 255:red, 0; green, 0; blue, 0 }  ][line width=0.75]    (10.93,-3.29) .. controls (6.95,-1.4) and (3.31,-0.3) .. (0,0) .. controls (3.31,0.3) and (6.95,1.4) .. (10.93,3.29)   ;
\draw    (547.08,145) -- (595.87,100.84) ;
\draw [shift={(597.36,99.5)}, rotate = 137.86] [color={rgb, 255:red, 0; green, 0; blue, 0 }  ][line width=0.75]    (10.93,-3.29) .. controls (6.95,-1.4) and (3.31,-0.3) .. (0,0) .. controls (3.31,0.3) and (6.95,1.4) .. (10.93,3.29)   ;
\draw  [dash pattern={on 4.5pt off 4.5pt}]  (70.7,54.5) -- (152.51,54) ;
\draw  [dash pattern={on 4.5pt off 4.5pt}]  (70.7,117) -- (152.51,116.5) ;
\draw  [dash pattern={on 4.5pt off 4.5pt}]  (552.19,54) -- (634,53.5) ;
\draw  [dash pattern={on 4.5pt off 4.5pt}]  (552.19,116.5) -- (634,116) ;

\draw (317.15,77.4) node [anchor=north west][inner sep=0.75pt]    {$\mathrm{Span}( \Phi _{0})$};
\draw (249.95,14.4) node [anchor=north west][inner sep=0.75pt]    {$\Omega $};
\draw (131.29,74.4) node [anchor=north west][inner sep=0.75pt]    {$P_{+}$};
\draw (553.19,78.4) node [anchor=north west][inner sep=0.75pt]    {$P_{-}$};
\draw (184.98,145.4) node [anchor=north west][inner sep=0.75pt]    {$E_{+}^{'}$};
\draw (525.07,145.4) node [anchor=north west][inner sep=0.75pt]    {$E_{-}^{'}$};

\end{tikzpicture}

\caption{Schematic for the construction of $N'$ when $\mathrm{Span}(\Phi_0) \subset \Omega$.}
\label{fig:sub1}
\end{figure}
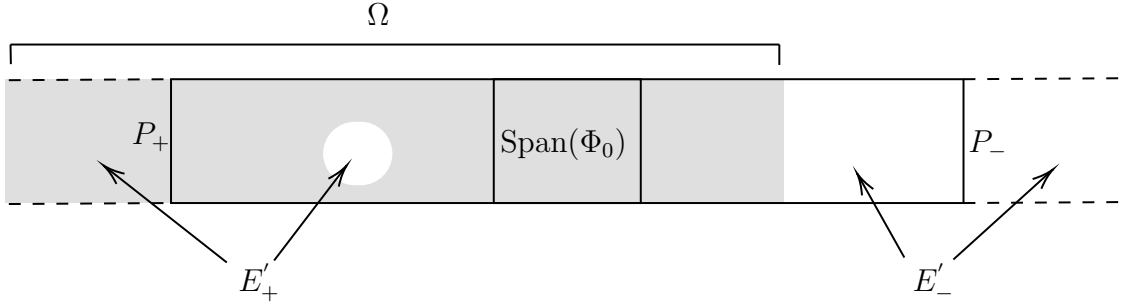

We consider two possibilities.
\begin{itemize}
    \item $\Span(\Phi_0) \subset \Omega$. Consider the subdomain $N' = N \cap \overline{\Omega}$. Then $\partial N' = P_+ \sqcup \partial \Omega$. We define $E_+'$ to be the union of $E_+$ and all components of $\tilde{N} \setminus N'$ in $I_+$. $E_-'$ is defined to be the union of the remaining components of $\tilde{N} \setminus N'$. See Figure \ref{fig:sub1}.
    \item $\Span(\Phi_0) \cap \overline{\Omega} = \varnothing$. Consider the subdomain $N' = N \setminus \Omega$. Then $\partial N' = P_- \sqcup \partial \Omega$. We define $E_-'$ to be the union of $E_-$ and all components of $\tilde{N}\setminus N'$ in $I_-$. $E_+'$ is defined to be the union of the remaining components of $\tilde{N} \setminus N'$. See Figure \ref{fig:sub2}.
\end{itemize}
We record a few observations that apply in both cases.
\begin{itemize}
    \item $\mathcal{C}_{E_+', E_-'}(N') \subset \cCEE(N)$. Hence, $\cA^h$ (as defined on $\cCEE(N)$) is well-defined on $\mathcal{C}_{E_+', E_-'}(N')$.
    \item By construction, the image of $\Phi_0$ lies in $\mathcal{C}_{E_+', E_-'}(N')$. Let $\Pi(N')$ denote the $(X, Z)$-homotopy class of $\Phi_0$ for maps to $\mathcal{C}_{E_+', E_-'}(N')$, which satisfies $\Pi \supset \Pi(N')$ by definition.
    \item Let $\mathbf{L}^h(\Pi(N'))$ denote the $h$-width of $\Pi(N')$. Then $\mathbf{L}^h(\Pi(N')) \geq \mathbf{L}^h(N)$.
\end{itemize}

\begin{figure}      

\begin{tikzpicture}[x=0.75pt,y=0.75pt,yscale=-1,xscale=1]

\fill  [fill=gray!25!white ] (70.7,54.5) -- (267,54.5) -- (267,117) -- (70.7,117) -- cycle ;
\fill  [fill=gray!25!white ] (428.21,90.5) .. controls (428.21,81.66) and (435.37,74.5) .. (444.21,74.5) -- (447.15,74.5) .. controls (455.99,74.5) and (463.15,81.66) .. (463.15,90.5) -- (463.15,90.5) .. controls (463.15,99.34) and (455.99,106.5) .. (447.15,106.5) -- (444.21,106.5) .. controls (435.37,106.5) and (428.21,99.34) .. (428.21,90.5) -- cycle ;
\draw   (152.51,54) -- (552.19,54) -- (552.19,116.5) -- (152.51,116.5) -- cycle ;
\draw   (315.28,54) -- (389.42,54) -- (389.42,116.5) -- (315.28,116.5) -- cycle ;
\draw    (70.56,36.5) -- (265,36.5) ;
\draw    (70.56,36.5) -- (70.64,46.4) ;
\draw    (265,36.5) -- (265,46.8) ;
\draw    (199,141.5) -- (211.41,101.41) ;
\draw [shift={(212,99.5)}, rotate = 107.2] [color={rgb, 255:red, 0; green, 0; blue, 0 }  ][line width=0.75]    (10.93,-3.29) .. controls (6.95,-1.4) and (3.31,-0.3) .. (0,0) .. controls (3.31,0.3) and (6.95,1.4) .. (10.93,3.29)   ;
\draw    (180.64,146.5) -- (120,98.74) ;
\draw [shift={(118.43,97.5)}, rotate = 38.23] [color={rgb, 255:red, 0; green, 0; blue, 0 }  ][line width=0.75]    (10.93,-3.29) .. controls (6.95,-1.4) and (3.31,-0.3) .. (0,0) .. controls (3.31,0.3) and (6.95,1.4) .. (10.93,3.29)   ;
\draw    (522.36,144) -- (445.72,98.52) ;
\draw [shift={(444,97.5)}, rotate = 30.68] [color={rgb, 255:red, 0; green, 0; blue, 0 }  ][line width=0.75]    (10.93,-3.29) .. controls (6.95,-1.4) and (3.31,-0.3) .. (0,0) .. controls (3.31,0.3) and (6.95,1.4) .. (10.93,3.29)   ;
\draw    (547.08,145) -- (595.87,100.84) ;
\draw [shift={(597.36,99.5)}, rotate = 137.86] [color={rgb, 255:red, 0; green, 0; blue, 0 }  ][line width=0.75]    (10.93,-3.29) .. controls (6.95,-1.4) and (3.31,-0.3) .. (0,0) .. controls (3.31,0.3) and (6.95,1.4) .. (10.93,3.29)   ;
\draw  [dash pattern={on 4.5pt off 4.5pt}]  (70.7,54.5) -- (152.51,54) ;
\draw  [dash pattern={on 4.5pt off 4.5pt}]  (70.7,117) -- (152.51,116.5) ;
\draw  [dash pattern={on 4.5pt off 4.5pt}]  (552.19,54) -- (634,53.5) ;
\draw  [dash pattern={on 4.5pt off 4.5pt}]  (552.19,116.5) -- (634,116) ;
\draw    (284,21.5) -- (168.99,30.35) ;
\draw [shift={(167,30.5)}, rotate = 355.6] [color={rgb, 255:red, 0; green, 0; blue, 0 }  ][line width=0.75]    (10.93,-3.29) .. controls (6.95,-1.4) and (3.31,-0.3) .. (0,0) .. controls (3.31,0.3) and (6.95,1.4) .. (10.93,3.29)   ;
\draw    (314,21) -- (438.23,85.58) ;
\draw [shift={(440,86.5)}, rotate = 207.47] [color={rgb, 255:red, 0; green, 0; blue, 0 }  ][line width=0.75]    (10.93,-3.29) .. controls (6.95,-1.4) and (3.31,-0.3) .. (0,0) .. controls (3.31,0.3) and (6.95,1.4) .. (10.93,3.29)   ;

\draw (317.15,77.4) node [anchor=north west][inner sep=0.75pt]    {$\mathrm{Span}( \Phi _{0})$};
\draw (291.03,11.4) node [anchor=north west][inner sep=0.75pt]    {$\Omega $};
\draw (131.29,74.4) node [anchor=north west][inner sep=0.75pt]    {$P_{+}$};
\draw (553.19,78.4) node [anchor=north west][inner sep=0.75pt]    {$P_{-}$};
\draw (184.98,145.4) node [anchor=north west][inner sep=0.75pt]    {$E_{+}^{'}$};
\draw (525.07,145.4) node [anchor=north west][inner sep=0.75pt]    {$E_{-}^{'}$};

\end{tikzpicture}
\caption{Schematic for the construction of $N'$ when $\mathrm{Span}(\Phi_0) \cap \overline{\Omega} = \varnothing$.}
\label{fig:sub2}
\end{figure}
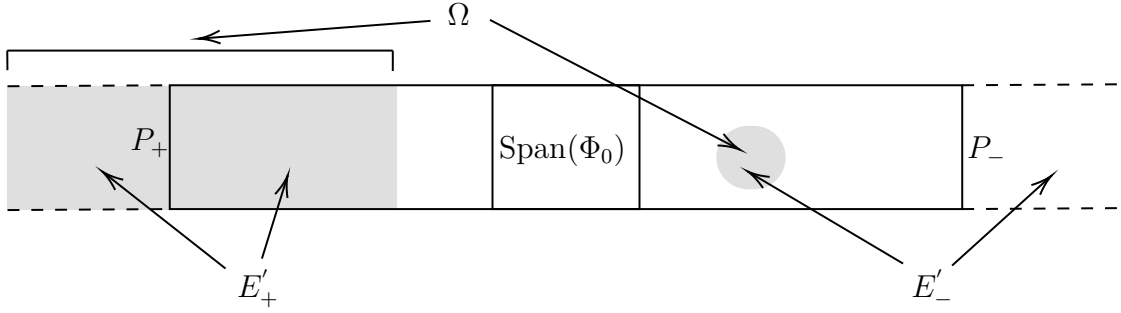

\subsection{Neighborhood cutting}\label{subsec:neighborhood_cutting}

We use the formula for the derivative of mean curvature along a smooth variation to find good hypersurfaces near unstable and strictly stable $h$-PMC hypersurfaces along which we will cut our manifold. 

\begin{lemma}\label{lem:neighborhood_cutting1}
    Let $(N, \partial N, g)$ be a compact Riemannian manifold with smooth boundary and $h \in C^{\infty}(N)$. Suppose $\Omega \in \cCEE(N)$ has smooth closed embedded boundary contained in $\interior(N)$. Suppose $\Sigma$ is a connected component of $\partial \Omega$ satisfying $H_{\Sigma, \Omega} = h\mid_{\Sigma}$ and $\Sigma$ is $\cA^h$-unstable. Then for any $\eps > 0$, there are sets $\Omega_- \subset \Omega$ and $\Omega_+ \supset \Omega$ satisfying
    \begin{itemize}
        \item $\Sigma \cap \Omega_- = \varnothing$ and $\Sigma \subset \Omega_+$,
        \item $\Omega \triangle \Omega_{\pm} \subset B_{\eps}(\Sigma)$,
        \item $\Sigma_{\pm} = (\partial \Omega_{\pm}) \cap B_{\eps}(\Sigma)$ are smooth closed embedded hypersurfaces contained in $\interior(N)$,
        \item $H_{\Sigma_-, \Omega_-} > h\mid_{\Sigma_-}$,
        \item $H_{\Sigma_+,\Omega_+} < h\mid_{\Sigma_+}$.
    \end{itemize}
\end{lemma}
\begin{proof}
    Let $\nu$ denote the unit normal vector field along $\Sigma$ pointing out of $\Omega$. Since $\Sigma = \partial \Omega$ is $\cA^h$-unstable, there is a negative number $\lambda < 0$ and a smooth positive function $\varphi \in C^{\infty}(\Sigma)$ so that
    \begin{equation}\label{eqn:jacobi}
        -\Delta_{\Sigma} \varphi - (|A_{\Sigma}|^2 + \Ric(\nu,\nu))\varphi = (\partial_\nu h)\varphi + \lambda \varphi < (\partial_\nu h)\varphi.
    \end{equation}
    Consider a variation of $\Sigma$ given by $\varphi\nu$ (namely, extend $\varphi\nu$ to a vector field supported in a neighborhood of $\Sigma$, then apply the flow diffeomorphisms generated by this vector field to $\Sigma$ and $\Omega$). Since $\varphi$ is positive and $\Sigma$ is closed and embedded, we obtain a foliation $\{\Sigma_t = \partial \Omega_t\}_{t \in (-\delta, \delta)}$ of a neighborhood of $\Sigma$ for some small $\delta > 0$, where $\Sigma_0 = \Sigma$ and $\Omega_0 = \Omega$. 
    
    Since the derivative of the mean curvature along this foliation at $t=0$ is given by the left hand side of
    \eqref{eqn:jacobi} and the derivative of $h$ along this foliation at $t=0$ is given by the right hand side of \eqref{eqn:jacobi}, we find that for all $\eta > 0$ sufficiently small
    \[ H_{\Sigma_{-\eta}, \Omega_{-\eta}} > h\mid_{\Sigma_{-\eta}}\ ,\ \ H_{\Sigma_{\eta}, \Omega_\eta} < h\mid_{\Sigma_\eta}. \]
    We can now take $\Omega_{\pm} = \Omega_{\pm\eta}$ for $\eta > 0$ chosen sufficiently small depending on $\eps$.
\end{proof}

\begin{lemma}\label{lem:neighborhood_cutting2}
    Let $(N, \partial N, g)$ be a compact Riemannian manifold with smooth boundary and $h \in C^{\infty}(N)$. Suppose $\Omega \in \cCEE(N)$ has smooth closed embedded boundary contained in $\interior(N)$. Suppose $\Sigma$ is a connected component of $\partial \Omega$ satisfying $H_{\Sigma, \Omega} = h\mid_{\Sigma}$ and $\Sigma$ is strictly $\cA^h$-stable. Then for any $\eps > 0$, there are sets $\Omega_- \subset \Omega$ and $\Omega_+ \supset \Omega$ satisfying
    \begin{itemize}
        \item $\Sigma \cap \Omega_- = \varnothing$ and $\Sigma \subset \Omega_+$,
        \item $\Omega \triangle \Omega_{\pm} \subset B_{\eps}(\Sigma)$,
        \item $\Sigma_{\pm} = (\partial \Omega_{\pm}) \cap B_{\eps}(\Sigma)$ are smooth closed embedded hypersurfaces contained in $\interior(N)$,
        \item $H_{\Sigma_-, \Omega_-} < h\mid_{\Sigma_-}$,
        \item $H_{\Sigma_+,\Omega_+} > h\mid_{\Sigma_+}$.
    \end{itemize}
\end{lemma}
\begin{proof}
    The proof is identical to Lemma \ref{lem:neighborhood_cutting1}, where we flip the direction of the inequality \eqref{eqn:jacobi} due to strict stability.
\end{proof}

\subsection{Localized min-max theorem with barrier}
We are now equipped to prove a localized version of Theorem \ref{thm:basic_min-max}.

\begin{theorem}\label{thm:localized_compact}
    Suppose $\partial N$ is a barrier for $(h, P_+, P_-)$. Assume
    \[ \mathbf{L}^h(\Pi) > \sup_{z \in Z} \cA^h(\Phi_0(z)). \]
    
    Then there is a smooth, closed, two-sided, almost embedded $h$-PMC hypersurface $\Sigma$ contained in $\mathrm{Int}(N)$ satisfying
    \begin{itemize}
        \item $\cH^n(\Sigma) \leq \sup_{x \in X} \cA^h(\Phi_0(x)) + \int_{\{h > 0\}} h$,
        \item $\mathrm{index}_{\cA^h}(\Sigma) \leq k$,
        \item $\Sigma \cap \Loc(\Phi_0, h) \neq \varnothing$.
    \end{itemize}
\end{theorem}
\begin{proof}
    We give the proof when $(g, h)$ is a good pair satisfying
    \[ h\mid_{I_+(\Phi_0) \setminus \Loc(\Phi_0, h)} > 0\ \ \text{and}\ \ h\mid_{I_-(\Phi_0) \setminus \Loc(\Phi_0, h)} < 0. \]
    By Theorem \ref{thm:generic} and \cite[Theorem 2.8]{zhou_mult}, the general conclusion follows by approximation\footnote{To ensure $h$ has the correct sign on the respective sets, we can first add to $h$ the function $\eps\eta$ for $\eps > 0$ small and $\eta \in C^{\infty}(N)$ any smooth function satisfying $\eta\mid_{\overline{I}_+(\Phi_0)} > 0$ and $\eta\mid_{\overline{I}_-(\Phi_0)} < 0$, and then take $\eps \to 0$.}.

    By Theorem \ref{thm:generic_min-max}, there is $\Omega \in \cCEE(N)$ so that $\cA^h(\Omega) = \mathbf{L}^h(\Pi)$ and $\Sigma = \partial \Omega$ is a smooth, closed, embedded, $h$-PMC hypersurface contained in $\interior(N)$, satisfying $1 \leq \mathrm{index}_{\cA^h}(\Sigma) \leq k$.

    If $\Sigma \cap \Loc(\Phi_0, h) \neq \varnothing$, then we are done.

    First, suppose that $\Loc(\Phi_0, h) \subset \Omega$. We modify $\Omega$ in a small neighborhood of $\Sigma$ in the following way. Let $\Sigma^*$ be a component of $\Sigma$, which (by bumpiness) is either unstable or strictly stable for $\cA^h$.
    \begin{enumerate}
        \item If $\Sigma^* \subset I_+$, then $h\mid_{\Sigma^*} > 0$. In this case, we modify $\Omega$ near $\Sigma^*$ by taking $\Omega_-$ from Lemma \ref{lem:neighborhood_cutting1} or Lemma \ref{lem:neighborhood_cutting2}, choosing the perturbation small enough so that the modified boundary $\Sigma_-$ has $H_{\Sigma_-, \Omega_-} > 0$.
        \item If $\Sigma^*$ is unstable and $\Sigma^* \subset I_-$, then we modify $\Omega$ near $\Sigma^*$ by taking $\Omega_-$ from Lemma \ref{lem:neighborhood_cutting1}.
        \item If $\Sigma^*$ is strictly stable and $\Sigma^* \subset I_-$, then we modify $\Omega$ near $\Sigma^*$ by taking $\Omega_+$ from Lemma \ref{lem:neighborhood_cutting1}.
    \end{enumerate}
    Let $\Omega'$ denote the modification of $\Omega$. We take the subdomain $N' = N \cap \overline{\Omega}'$, with the decomposition $E_{\pm}'$ given in \S\ref{subsec:subdomains}. By construction, for $\Sigma^*$ in case (1) above, the modified boundary $\Sigma_-$ will lie in $P_+'$. Then we have
    \[ H_{\Sigma_-, E_+'} < 0 < h\mid_{\Sigma_-}, \]
    so these boundary components satisfy the barrier property. For $\Sigma^*$ in cases (2) and (3), the modified boundary will lie in $P_-'$. Then the barrier property follows directly from Lemmas \ref{lem:neighborhood_cutting1} and \ref{lem:neighborhood_cutting2} respectively.

    Second, suppose that $\Loc(\Phi_0, h) \cap \overline{\Omega} = \varnothing$. We modify $\Omega$ in a small neighborhood of $\Sigma$ in the following way. Let $\Sigma^*$ be a component of $\Sigma$, which (by bumpiness) is either unstable or strictly stable for $\cA^h$.
    \begin{enumerate}
        \item If $\Sigma^* \subset I_-$, then $h\mid_{\Sigma^*} < 0$. In this case, we modify $\Omega$ near $\Sigma^*$ by taking $\Omega_+$ from Lemma \ref{lem:neighborhood_cutting1} or Lemma \ref{lem:neighborhood_cutting2}, choosing the perturbation small enough so that the modified boundary $\Sigma_+$ has $H_{\Sigma_+, \Omega_+} < 0$.
        \item If $\Sigma^*$ is unstable and $\Sigma^* \subset I_+$, then we modify $\Omega$ near $\Sigma^*$ by taking $\Omega_+$ from Lemma \ref{lem:neighborhood_cutting1}.
        \item If $\Sigma^*$ is strictly stable and $\Sigma^* \subset I_+$, then we modify $\Omega$ near $\Sigma^*$ by taking $\Omega_-$ from Lemma \ref{lem:neighborhood_cutting2}.
    \end{enumerate}
    Let $\Omega'$ denote the modification of $\Omega$. We take the subdomain $N' = N \setminus \interior(\Omega')$, with the decomposition $E_{\pm}'$ given in \S\ref{subsec:subdomains}. By construction, for $\Sigma^*$ in case (1) above, the modified boundary $\Sigma_+$ will lie in $P_-'$. Then we have
    \[ H_{\Sigma_+, E_-'} < 0 < -h\mid_{\Sigma_-}, \]
    so these boundary components satisfy the barrier property. For $\Sigma^*$ in cases (2) and (3), the modified boundary will lie in $P_+'$. Then the barrier property follows directly from Lemmas \ref{lem:neighborhood_cutting1} and \ref{lem:neighborhood_cutting2} respectively.
    
    In both cases, we find a subdomain $N'$ with decomposition $E_{\pm}'$ so that
    \begin{itemize}
        \item $\Span(\Phi_0) \subset \Loc(\Phi_0, h) \subset N'$,
        \item $\partial N'$ is a barrier for $(h, P_+', P_-')$,
        \item the unstable components of $\Sigma$ (of which there is at least one by Theorem \ref{thm:generic_min-max}) are not contained in $N'$,
        \item $(g\mid_{N'}, h\mid_{N'})$ is a good pair\footnote{It follows from the definition that the restriction of a good pair to a subdomain is a good pair.},
        \item $\sup_{x \in X} \cA^h(\Phi_0(x)) \geq \mathbf{L}^h(\Pi(N')) \geq \mathbf{L}^h(\Pi) > \sup_{z\in Z} \cA^h(\Phi_0(z))$.
    \end{itemize}

    Hence, we can re-apply Theorem \ref{thm:generic_min-max} on $N'$, and therefore iterate the above steps. Since each iteration produces a distinct, connected, $\cA^h$-unstable, smooth, closed, embedded, $h$-PMC hypersurface $\Sigma'$ with
    \[ \cH^n(\Sigma') \leq \sup_{x \in X} \cA^h(\Phi_0(x)) + \int_{\{h > 0\}} h < \infty \]
    and $\mathrm{index}_{\cA^h}(\Sigma') \leq k$, the iteration must terminate after at most finitely many steps by the bumpiness of the metric (if there were infinitely many such hypersurfaces, \cite[Theorem 2.6(iv)]{zhou_mult} yields a limit embedded $h$-PMC hypersurface that is degenerate). Hence, at some step of the iteration the output of Theorem \ref{thm:generic_min-max} will satisfy $\Sigma \cap \Loc(\Phi_0, h) \neq \varnothing$.
\end{proof}

\part{Noncompact Manifolds}\label{part:noncompact}
In this part, we apply the localized min-max result of Theorem \ref{thm:localized_compact} to develop a min-max theory in noncompact manifolds. We also provide applications of the theory.

\section{Setup}\label{sec:noncompact_setup}

Let $(M, g)$ be a complete Riemannian manifold of dimension $3 \leq n+1 \leq 7$.

\subsection{Space of subsets}
Let $\Omega_0$ be an open set with smooth closed boundary (we emphasize that $\Omega_0$ need not be bounded).

Let $\mathcal{C}_{\Omega_0}(M)$ be the collection of sets of finite perimeter $\Omega \in \mathcal{C}(M)$ so that $\Omega \triangle \Omega_0$ is bounded.

\subsection{Functional}\label{subsec:noncompact_functional}
Let $h \in C^{\infty}_{\text{loc}}(M)$. For $\Omega \in \mathcal{C}_{\Omega_0}(M)$, we define
\[ \cA^h(\Omega) = \cH^n(\partial^*\Omega) - \int_{\Omega \setminus \Omega_0} h + \int_{\Omega_0 \setminus \Omega} h. \]
By the definition of $\mathcal{C}_{\Omega_0}(M)$, $\cA^h(\Omega)$ is well-defined for all $\Omega \in \mathcal{C}_{\Omega_0}(M)$.

\subsection{Admissible functions}
We let $\mathcal{W}_{\Omega_0} \subset C_{\text{loc}}^\infty(M)$ denote the set of locally smooth functions $h$ with the property that there is a compact set $K_h \subset M$ satisfying
\[ h\mid_{\Omega_0 \setminus K_h} \geq 0\ \ \text{and}\ \ h\mid_{(M\setminus \Omega_0)\setminus K_h} \leq 0. \]
Note that if $\Omega \in \mathcal{C}_{\Omega_0}(M)$, $h \in \mathcal{W}_{\Omega_0}$, and $\cA^h(\Omega) \leq \Lambda < \infty$, then
\begin{align*}
    \cH^n(\partial^* \Omega)
    & = \cA^h(\Omega) + \int_{\Omega \setminus \Omega_0} h - \int_{\Omega_0 \setminus \Omega} h\\
    & \leq \Lambda + \int_{(\Omega \setminus \Omega_0) \cap K_h} h - \int_{(\Omega_0 \setminus \Omega) \cap K_h} h\\
    & \leq \Lambda + \cH^{n+1}(K_h)\sup_{K_h} |h| < \infty.
\end{align*}
Hence, for these functions, uniform $\cA^h$ bounds imply uniform area bounds.

\subsection{Min-max setup}
We follow the same definitions as \S\ref{subsec:min-max_setup}.

Let $X^k$ be a connected cubical complex of dimension $k$ in some $I^m = [0, 1]^m$. Let $Z \subset X$ be a nonempty cubical subcomplex.

Let $\Phi_0 : X \to \mathcal{C}_{\Omega_0}(M)$ be a continuous map in the $\mathbf{F}$-topology.

\begin{definition}[Homotopy class]
    Let $\Pi$ be the set of all sequences of $\mathbf{F}$-continuous maps $\{\Phi_i : X \to \mathcal{C}_{\Omega_0}(M)\}_i$ so that there are $\mathcal{F}$-continuous homotopies $\{\Psi_i : [0, 1] \times X \to \mathcal{C}_{\Omega_0}(M)\}_i$ satisfying $\Psi_i(0, \cdot) = \Phi_i$, $\Psi_i(1, \cdot) = \Phi_0$, and
    \[ \limsup_{i\to\infty} \sup\{\mathbf{F}(\Psi_i(t, z), \Phi_0(z)) \mid t \in [0, 1],\ z \in Z\} = 0. \]
    Such a sequence $\{\Phi_i\}$ is called an \emph{$(X, Z)$-homotopy sequence of mappings to $\mathcal{C}_{\Omega_0}(M)$}, and $\Pi$ is called the \emph{$(X, Z)$-homotopy class of $\Phi_0$}.
\end{definition}

\begin{definition}[Width]
    The \emph{$h$-width} of $\Pi$ is
    \[ \mathbf{L}^h(\Pi) = \inf_{\{\Phi_i\} \in \Pi} \limsup_{i \to \infty} \sup_{x \in X} \cA^h(\Phi_i(x)). \]
\end{definition}

\begin{definition}[Minimizing sequence]
    We say $\{\Phi_i\}_i \in \Pi$ is a \emph{minimizing sequence} if
    \[ \limsup_{i \to \infty} \sup_{x \in X} \cA^h(\Phi_i(x)) = \mathbf{L}^h(\Pi). \]
\end{definition}

\begin{definition}[Critical set]
    The \emph{critical set} of a minimizing sequence $\{\Phi_i\}_i \in \Pi$ is
    \[ \mathbf{C}(\{\Phi_i\}) = \left\{V \in \mathcal{V}_n(M) \mid V = \lim_{j\to \infty} |\partial^*\Phi_{i_j}(x_j)|,\ \lim_{j \to \infty} \cA^h(\Phi_{i_j}(x_j)) = \mathbf{L}^h(\Phi)\right\}. \]
\end{definition}

\section{Localization in noncompact manifolds}\label{sec:localization_noncompact}

We use our localized min-max theory in compact manifolds with barriers to deduce a localized min-max theory in noncompact manifolds.

\subsection{Setup}
We follow the same definitions as \S\ref{subsec:local_setup}.

\begin{definition}[Span]
    The \emph{span} of $\Phi_0$ is 
    \[ \Span(\Phi_0) = \overline{\bigcup_{x \in X} \partial^*\Phi_0(x)}. \]
\end{definition}

\begin{definition}[Unspanned halves]
    For any $x_0 \in X$\footnote{Recall that we can assume (without loss of generality) that $\partial \Phi_0(x) = \overline{\partial^*\Phi_0(x)}$ for all $x \in X$.}, we define the positive and negative \emph{unspanned halves} by
    \[ I_+(\Phi_0) = (M \cap \interior(\Phi_0(x_0))) \setminus \Span(\Phi_0)\ \ \text{and}\ \ I_-(\Phi_0) = (M \setminus \overline{\Phi_0(x_0)}) \setminus \Span(\Phi_0). \]
    Note that $I_{\pm}(\Phi_0)$ are well-defined: for any $x \in X$ we have
    \[ I_+(\Phi_0) \subset \interior(\Phi_0(x))\ \ \text{and}\ \ I_-(\Phi_0) \cap \overline{\Phi_0(x)} = \varnothing. \]
\end{definition}

\begin{definition}[Localizer]
    The \emph{localizer} of $\Phi_0$ and $h$ is
    \[ \Loc(\Phi_0, h) = \Span(\Phi_0) \cup (\overline{I_+(\Phi_0)} \cap \overline{\{h < 0\}}) \cup (\overline{I_-(\Phi_0)} \cap \overline{\{h > 0\}}). \]
\end{definition}

\begin{remark}
    For $h \in \mathcal{W}_{\Omega_0}$, the localizer $\Loc(\Phi_0, h)$ is a compact set.
\end{remark}

\subsection{Compact subdomains}\label{subsec:compact_subdomains}
Here we discuss how to restrict the min-max setup to special compact subdomains of the original complete manifold, where we assume $h \in \mathcal{W}_{\Omega_0}$.

Let $N \subset M$ be a compact subdomain with smooth boundary satisfying $\mathrm{Loc}(\Phi_0, h) \subset N$.

We give $N$ a boundary decomposition as in \S\ref{subsec:subsets} by
\[ P_+ = \partial N \cap I_+(\Phi_0)\ \ \text{and}\ \ P_- = \partial N \cap I_-(\Phi_0). \]
We also set
\[ F_+ = I_+(\Phi_0) \setminus N \ \ \text{and}\ \ F_- = I_-(\Phi_0) \setminus N. \]
Let $\mathcal{C}_{F_+, F_-}(M)$ be the space of $\Omega \in \mathcal{C}_{\Omega_0}(M)$ satisfying
\[ F_+ \subset \Omega \ \ \text{and}\ \ \Omega \cap F_- = \varnothing. \]

Let $(\tilde{N}, \tilde{g})$ be a closed Riemannian manifold of dimension $n+1$ containing $(N, g)$ isometrically, and so that $\tilde{N}\setminus N$ has a decomposition $E_+ \sqcup E_-$ into unions of connected components of $\tilde{N} \setminus N$ satisfying $\partial E_{\pm} = P_{\pm}$.

We are now in the setting of Part \ref{part:compact}. We make a few observations.
\begin{itemize}
    \item There is a natural $\mathbf{F}$-homeomorphism from $\mathcal{C}_{F_+, F_-}(M)$ to $\cCEE(N)$ given by deleting $F_+$ from $\Omega$ and adding $E_+$.
    \item Using this homeomorphism, $\cA^h$ as defined in \S\ref{subsec:noncompact_functional} is well-defined on $\cCEE(N)$. Moreover, this definition of $\cA^h$ only differs from \S\ref{subsec:prescription} by a fixed constant, which does not affect the theory in Part \ref{part:compact} (since the difference of $\cA^h$ between any two sets is unchanged).
    \item By construction, the image of $\Phi_0$ under this homeomorphism lies in $\mathcal{C}_{E_+, E_-}(N)$. Let $\Pi(N)$ denote the $(X, Z)$-homotopy class of $\Phi_0$ for maps to $\mathcal{C}_{E_+, E_-}(N)$, which satisfies $\Pi \supset \Pi(N)$ by definition.
    \item Let $\mathbf{L}^h(\Pi(N))$ denote the $h$-width of $\Pi(N)$. Then $\mathbf{L}^h(\Pi(N)) \geq \mathbf{L}^h(N)$.
\end{itemize}

\subsection{Localized min-max theorem in noncompact manifolds}

In this subsection, we prove the main min-max theorem in noncompact manifolds.

\begin{theorem}\label{thm:localized_noncompact}
    Suppose $h \in \mathcal{W}_{\Omega_0}$. Suppose
    \[ \mathbf{L}^h(\Pi) > \sup_{z \in Z} \cA^h(\Phi_0(z)). \]
    
    Then there is a smooth, complete, almost embedded, $h$-PMC hypersurface $\Sigma$ in $M$ satisfying
    \begin{itemize}
        \item $\cH^n(\Sigma) \leq \sup_{x\in X}\cA^h(\Phi_0(x)) + \cH^{n+1}(K_h)\sup_{K_h}|h| < \infty$,
        \item $\mathrm{index}_{\cA^h}(\Sigma) \leq k$,
        \item $\Sigma \cap \Loc(\Phi_0, h) \neq \varnothing$.
    \end{itemize}
\end{theorem}
\begin{proof}
    Let $\{N_i\}_i$ be a nested compact exhaustion of $M$ with smooth boundary satisfying $\Loc(\Phi_0, h) \subset N_1$.

    Fix any $i$ (henceforth we drop the subscript). We follow the constructions of \S\ref{subsec:compact_subdomains} for $N$, which produces a boundary decomposition $P_{\pm}$.

    We modify $h$ to $\tilde{h}$ near $\partial N$ so that $\tilde{h}\mid_{P_+} > H_{P_+, M \setminus N}$, $-\tilde{h}\mid_{P_-} > H_{P_-, M \setminus N}$, and $\int_{\tilde{N}} |h - \tilde{h}|$ is arbitrarily small.

    Then $\partial N$ is a barrier for $(\tilde{h}, P_+, P_-)$ and $\mathbf{L}^{\tilde{h}}(\Pi(N)) > \sup_{z\in Z} \cA^{\tilde{h}}(\Phi_0(z))$.

    We can now apply Theorem \ref{thm:localized_compact}, producing a smooth closed almost embedded $\tilde{h}$-PMC hypersurface $\Sigma$ in $\interior(N)$ satisfying the conclusions of Theorem \ref{thm:localized_compact}\footnote{The difference in the area bounds is simply a reflection of the fact that the definition of $\cA^h$ here differs from the definition of $\cA^h$ in Part \ref{part:compact} by a fixed constant.}.

    The theorem follows by taking the limit over the compact exhaustion and applying the compactness theory from \cite[Theorem 2.8]{zhou_mult}.
\end{proof}

\section{Application I: PMC hypersurfaces in Euclidean space}\label{sec:app1}

We use the theory developed in Theorem \ref{thm:basic_min-max} (and Corollary \ref{cor:basic_min-max}) to provide new answers to Question \ref{quest:PMCrn}.

\begin{theorem}\label{thm:rn}
    If $h \in C^{\infty}_{\text{loc}}(\R^{n+1})$ for $3 \leq n+1 \leq 7$ satisfies
    \begin{enumerate}
        \item there is a compact set with smooth boundary $N \subset \R^{n+1}$ satisfying $h < H_{\partial N, N}$,
        \item there is an $\mathbf{F}$-continuous map $\Phi_0 : [0, 1] \to \mathcal{C}_{\varnothing}(\R^{n+1})$ satisfying $\Phi_0(t) \subset N$ for all $t$ and
        \[ \mathbf{L}^h(\Pi_{\Phi_0}(N)) > \max\{\cA^h(\Phi_0(0)),\ \cA^h(\Phi_0(1))\}. \]
    \end{enumerate}
    then there exists a smooth closed almost embedded $h$-PMC hypersurface contained in $\interior(N)$ with $\cA^h$-index at most 1.
\end{theorem}
\begin{remark}
    We emphasize that condition (1) is only a 1-sided bounded; $h$ is allowed to be arbitrarily negative.
\end{remark}
\begin{proof}
    Apply Theorem \ref{thm:localized_compact} (and Corollary \ref{cor:basic_min-max} for the index bound) on $N$ with $P_+ = \varnothing$ and $P_- = \partial N$.
\end{proof}

\begin{corollary}
    If $h \in C^{\infty}_{\text{loc}}(\R^{n+1})$ for $3 \leq n+1 \leq 7$ satisfies
    \begin{enumerate}
        \item there exists a nonempty bounded set of finite perimeter $\Lambda$ satisfying $\cA^h(\Lambda) \leq 0$,
        \item $h(x) \leq C|x|^{-1-\alpha}$ for some $\alpha,\ C > 0$,
    \end{enumerate}
    then there exists a smooth closed almost embedded $h$-PMC hypersurface in $\R^{n+1}$ with $\cA^h$-index at most 1.
\end{corollary}
\begin{proof}
    Choose $R$ so large that $\Lambda \subset B_R(0)$ and $h\mid_{\partial B_R(0)} \leq CR^{-1-\alpha} < n/R$. Set $N = \overline{B_R(0)}$, and note that $H_{\partial N, N} = n/R > h$.
    
    Let $\Phi_0(t) = \Lambda \cap B_{Rt}(0)$. Then
    \[ \max\{\cA^h(\Phi_0(0)), \cA^h(\Phi_0(1))\} = \max\{\cA^h(\varnothing), \cA^h(\Lambda)\} = 0. \]
    By the Euclidean isoperimetric inequality (and the fact that $h$ is bounded on $N$), we have $\mathbf{L}^h(\Phi_0(N)) > 0$.

    Hence, the conclusion follows from Theorem \ref{thm:rn}.
\end{proof}

\section{Application II: CMC hypersurfaces in finite volume}\label{sec:app2}

We use the theory developed in Theorem \ref{thm:localized_noncompact} to provide new answers to Question \ref{quest:CMCfv}.

\subsection{Isoperimetric profile}
Let $(M, g)$ be a complete finite volume manifold. 

\begin{definition}[Isoperimetric profile]
    For $v \in (0, \mathrm{Vol}(M))$, we define
    \[ \mathscr{I}(v) = \inf\{\cH^n(\partial^*\Omega) \mid \Omega \in \mathcal{C}(M),\ \cH^{n+1}(\Omega) = v\}. \]
\end{definition}
Since we assume $(M, g)$ has finite volume, $\mathscr{I}(v) > 0$ for all $v \in (0, \mathrm{Vol}(M))$ and the infimum is achieved by a finite perimeter set (by \cite[Theorem 2.2]{NR}), and $\mathscr{I}$ is continuous (by \cite[Corollary 2.3]{NR}).

\begin{definition}
    We define
    \[ \mathscr{D}(M, g) := \mathrm{Lip}(\mathscr{I}) = \sup_{v_1 \neq v_2} \left|\frac{\mathscr{I}(v_2) - \mathscr{I}(v_1)}{v_2-v_1}\right|. \]
\end{definition}

\begin{remark}
    Since $\mathscr{I}(v) > 0$ for finite volume manifolds, $\mathscr{D}(M, g) > 0$.
\end{remark}

\subsection{Existence of CMC hypersurfaces}

We now prove the main result of the subsection.

\begin{theorem}\label{thm:cmc}
    Let $(M, g)$ be a complete finite volume manifold of dimension $3 \leq n+1 \leq 7$. If $0 < c < \mathscr{D}(M, g)$, then there is a smooth complete almost embedded $c$-CMC hypersurface in $M$ with $\cA^c$-index at most 1.
\end{theorem}
\begin{proof}
    Since $\mathscr{I}$ is symmetric under $v \mapsto \mathrm{Vol}(M) - v$, we can find $v_1 < v_2$ so that
    \[ \frac{\mathscr{I}(v_1) - \mathscr{I}(v_2)}{v_2 - v_1} > c. \]
    By Lemma \ref{lem:approx_by_cpt} below, we can find $\Omega_1 \in \mathcal{C}_{\varnothing}(M)$ with $\cH^{n+1}(\Omega_1)$ approximately $v_2$ and
    \begin{equation}\label{eqn:cmc_application}
        \frac{\mathscr{I}(v_1) - \cH^n(\partial^* \Omega_1)}{\cH^{n+1}(\Omega_1) - v_1} \geq c + \delta
    \end{equation}
    for some $\delta > 0$ small.

    Let $\Phi_0 : [0, 1] \to \mathcal{C}_{\varnothing}(M)$ be any $\mathbf{F}$-continuous path from $\varnothing$ to $\Omega_1$. Take $h = -c$. Then $\varnothing$ is the unique minimizer of $\cA^h$. For any $\Omega \in \mathcal{C}_{\varnothing}(M)$ with $\cH^{n+1}(\Omega) = v_1$, we have by \eqref{eqn:cmc_application}
    \begin{align*}
        \cA^h(\Omega)
        & = \cH^n(\partial^*\Omega) + c\cH^{n+1}(\Omega)\\
        & \geq \mathscr{I}(v_1) + cv_1\\
        & \geq \cH^n(\partial^* \Omega_1) + c\cH^{n+1}(\Omega_1) + \frac{\delta}{2}(v_2 - v_1)\\
        & = \cA^h(\Omega_1) + \frac{\delta}{2}(v_2 - v_1).
    \end{align*}
    Hence, we have
    \[ \mathbf{L}^h(\Pi_{\Phi_0}) > \max\{\cA^h(\Phi_0(0)),\ \cA^h(\Phi_0(1))\}. \]
    The theorem now follows from Theorem \ref{thm:localized_noncompact}.
\end{proof}

\begin{lemma}\label{lem:approx_by_cpt}
    For any $v \in (0, \mathrm{Vol}(M))$ and any $\eps > 0$, there is a bounded set $\Omega \in \mathcal{C}(M)$ satisfying
    \[ |\cH^{n+1}(\Omega) - v| < \eps\ \ \text{and}\ \ |\mathscr{I}(v) - \cH^n(\partial^*\Omega)| < \eps. \]
\end{lemma}
\begin{proof}
    Since $M$ has finite volume, there is $\{K_i\}_i$ a nested exhaustion by compact sets with smooth boundary satisfying $\limsup_{i \to \infty} \cH^n(\partial K_i) = 0$.
    
    Let $\Omega \in \mathcal{C}(M)$ satisfy $\cH^{n+1}(\Omega) = v$ and $\cH^n(\partial^*\Omega) \leq \mathscr{I}(v) + \eps/2$. Let $\Omega_i = \Omega \cap K_i$. Then
    \[ \limsup_{i \to \infty} \cH^n(\partial^*\Omega_i) \leq  \cH^n(\partial^*\Omega) + \limsup_{i \to \infty} \cH^n(\partial K_i) = \cH^n(\partial^*\Omega), \]
    \[ \liminf_{i \to \infty} \cH^n(\partial^*\Omega_i) \geq \liminf_{i \to \infty} \cH^n((\partial^*\Omega) \cap K_i) = \cH^n(\partial^*\Omega). \]
    Moreover,
    \[ \limsup_{i \to \infty} \cH^{n+1}(\Omega_i) \leq \cH^{n+1}(\Omega), \]
    \[ \liminf_{i \to \infty} \cH^{n+1}(\Omega_i) \geq \cH^{n+1}(\Omega) - \limsup_{i \to \infty}\cH^{n+1}(M \setminus K_i) = \cH^{n+1}(\Omega). \]
    Hence, we can take $\Omega_i$ for $i$ sufficiently large (depending on $\eps$).
\end{proof}

\subsection{Lower bounds for Lipschitz constant of isoperimetric profile} In the rest of the section, we develop tools to compute lower bounds for the invariant $\mathscr{D}$. We conclude the subsection with a few examples.

\begin{proposition}\label{prop:dlowbound}
    Let $\delta \in (0, \mathrm{Vol}(M))$. Suppose there are locally smooth functions
    \[ \eta_1 : (0, \delta) \to (0, +\infty),\ \ \eta_2 : (0, \delta) \to (0, +\infty) \]
    satisfying $\lim_{t \to 0} \eta_j(t) = 0$ and $\mathscr{I}(v) \geq \min\{\eta_1(v), \eta_2(v)\}$ for all $v \in (0, \delta)$. Then
    \[ \mathscr{D}(M, g) \geq \min\left\{\liminf_{t \to 0} \eta_1'(t),\ \liminf_{t\to 0} \eta_2'(t)\right\}. \]
\end{proposition}
\begin{proof}
    Set $m = \min\left\{\liminf_{t \to 0} \eta_1'(t),\ \liminf_{t\to 0} \eta_2'(t)\right\}$ (unless the right hand side is $+\infty$, in which case we set $m = 1$ for convenience). If $m = 0$, then the conclusion is trivial, so we assume $m > 0$.

    Fix $\eps \in (0, m)$. We claim that
    \[ \min\{\eta_1(v), \eta_2(v)\} \geq (m-\eps)v. \]
    for all $v$ sufficiently small. Indeed, by the definition of $m$, there is some $\sigma > 0$ so that for all $t \in (0, \sigma)$, we have $\eta_j'(t) \geq m-\eps$ for $j = 1, 2$. By the fundamental theorem of calculus, we have $\eta_j(t) \geq (m-\eps)t$ for all $t \in (0, \sigma)$ for $j = 1, 2$, so the desired conclusion holds.

    Hence, we have $\mathscr{I}(v) \geq (m-\eps)v$ for any $v \in (0, \sigma)$. Let $v_2 \in (0, \sigma)$. Since $\lim_{v\to 0} \mathscr{I}(v) = 0$, there is a sequence of positive real numbers $\{v_1^i\}_i$ converging to $0$ with $\mathscr{I}(v_1^i) < \mathscr{I}(v_2)$. We have
    \begin{align*}
        \mathscr{D}(M, g)
        & \geq \liminf_{i \to \infty} \left|\frac{\mathscr{I}(v_2) - \mathscr{I}(v_1^i)}{v_2 - v_1^i}\right|\\
        & = \liminf_{i \to \infty}\frac{\mathscr{I}(v_2) - \mathscr{I}(v_1^i)}{v_2 - v_1^i}\\
        & \geq \liminf_{i \to \infty}\frac{\mathscr{I}(v_2) - \mathscr{I}(v_1^i)}{v_2}\\
        & \geq m-\eps.
    \end{align*}
    Since $\eps > 0$ was arbitrary, the conclusion follows.
\end{proof}

Suppose $E \subset M$ is the complement of a precompact open subset with smooth boundary in $M$. We define
\[ \mathscr{I}_E(v) = \inf\{\cH^n(\partial^*\Omega) \mid \Omega \in \mathcal{C}(M),\ \Omega \subset E,\ \cH^{n+1}(\Omega) = v\}. \]

\begin{proposition}\label{prop:noncompact_piece}
    Suppose $E \subset M$ is the complement of a precompact open subset with smooth boundary in $M$. Let $\delta \in (0, \cH^{n+1}(E))$. Suppose there is a locally smooth function
    \[ \eta : (0, \delta) \to (0, +\infty) \]
    satisfying $\lim_{t\to 0} \eta(t) = 0$ and $\mathscr{I}_E(v) \geq \eta(v)$ for all $v \in (0, \delta)$. Then
    \[ \mathscr{D}(M, g) \geq \liminf_{t \to 0} \eta'(t). \]
\end{proposition}
\begin{proof}
    If $\liminf_{t\to 0} \eta'(t) = 0$, then the conclusion is trivial, so we assume otherwise.
    
    Since $\partial E$ is smooth, compact, and separating, we can find a compact set $K \subset M$ with smooth boundary so that
    \begin{itemize}
        \item $E \cup K = M$,
        \item there is a diffeomorphism
        \[ \phi : \partial E \times [0, \eps] \to E \cap K \]
        satisfying $\phi(\partial E, t) = \{x \in E\mid d(x, \partial E) = t\}$.
    \end{itemize}
    Since $K$ is compact with smooth boundary, by \cite[Theorem 9.6]{iso} there are constants $C_K,\ v_K > 0$ so that for any $\Lambda \in \mathcal{C}(M)$ with $\cH^{n+1}(\Lambda \cap K) \leq v_K$, we have
    \begin{equation}\label{eqn:iso}
        \cH^n((\partial^*\Lambda) \cap K) \geq C_K(\cH^{n+1}(\Lambda \cap K))^{\frac{n}{n+1}}.
    \end{equation}
    For $t \in [0, \eps]$, we let $S_t = \phi(\partial E, t)$, $K_t = K \setminus \phi(\partial E, [t, \eps])$, and $E_t = E \setminus \phi(\partial E, [0, t])$.
    
    Fix $v \in (0, \delta)$. By \cite[Corollary 2.2]{NR} (since we assume $(M, g)$ has finite volume), we can find $\Omega \in \mathcal{C}(M)$ satisfying
    \[ \cH^{n+1}(\Omega) = v\ \ \text{and}\ \ \cH^n(\partial^*\Omega) = \mathscr{I}(v). \]

    Case 1: Suppose $\cH^{n+1}(\Omega \cap K) \geq v^{\frac{n+2}{n+1}}$. By \eqref{eqn:iso}, we have (for $v$ sufficiently small)
    \begin{equation}\label{eqn:case1}
        \mathscr{I}(v) = \cH^n(\partial^*\Omega) \geq \cH^n((\partial^*\Omega) \cap K) \geq C_Kv^{\frac{n+2}{n+1}\frac{n}{n+1}} = C_Kv^{\frac{n^2+2n}{n^2+2n+1}} =: \eta_1(v).
    \end{equation}

    Case 2: Suppose $\cH^{n+1}(\Omega \cap K) \leq v^{\frac{n+2}{n+1}}$. By the coarea formula in $E \cap K$, there is a constant $c_1 = c_1(E, K) > 0$ so that for some $t \in [0, \eps]$ (depending on $\Omega$), we have
    \begin{equation}\label{eqn:coarea}
        \cH^n(\Omega \cap S_t) \leq c_1v^{\frac{n+2}{n+1}}.
    \end{equation}
    By assumption and the definition of $\mathscr{I}_E$, we have
    \begin{equation}\label{eqn:inE}
        \cH^n((\partial^*\Omega) \cap E_t) + \cH^n(\Omega \cap S_t) = \cH^{n}(\partial^*(\Omega \cap E_t)) \geq \eta(\cH^{n+1}(\Omega \cap E_t)).
    \end{equation}
    Since $\liminf_{t \to 0}\eta'(t) > 0$, $\eta$ is increasing for $v$ sufficiently small. So for sufficiently small $v$, by \eqref{eqn:coarea} and \eqref{eqn:inE}, we have
    \begin{equation}\label{eqn:case2}
        \mathscr{I}(v) = \cH^n(\partial^*\Omega) \geq \cH^n((\partial^*\Omega) \cap E_t) \geq \eta(v - v^{\frac{n+2}{n+1}}) - c_1v^{\frac{n+2}{n+1}} =: \eta_2(v).
    \end{equation}

    Hence, for $v$ sufficiently small, we have
    \[ \mathscr{I}(v) \geq \min\{\eta_1(v), \eta_2(v)\}, \]
    where $\eta_1$ and $\eta_2$ are as defined in \eqref{eqn:case1} and \eqref{eqn:case2}.
    Since
    \[ \liminf_{t \to 0} \eta_1'(t) = +\infty \ \ \text{and}\ \  \liminf_{t \to 0} \eta_2'(t) = \liminf_{t \to 0} \eta'(t), \]
    the conclusion follows by Proposition \ref{prop:dlowbound}.
\end{proof}

We now compute lower bounds for $\mathscr{D}$ for a large class of finite volume manifolds with special structure outside a compact set.

\begin{proposition}\label{prop:compute}
    Let $(M^{n+1}, g)$ be a complete manifold with finite volume and dimension $3 \leq n+1 \leq 7$. Suppose there is a compact set $K \subset M$ so that $M \setminus K$ is isometric to
    \[ ((A, \infty) \times N,\ \psi^2(t)dt^2 \oplus \rho^2(t)g_N) \]
    for a fixed closed Riemannian manifold $(N^n, g_N)$ of dimension $n$ and locally smooth positive functions $\psi$ and $\rho$. If $\rho' < 0$ and $\frac{\rho'}{\rho\psi}$ is nonincreasing, then
    \[ \mathscr{D}(M, g) \geq -n\lim_{t \to \infty} \frac{\rho'(t)}{\rho(t)\psi(t)}. \]
\end{proposition}
\begin{proof}
    In the provided coordinates, we let $E = [A+1, \infty) \times N$.

    Take any $v \in (0, \cH^{n+1}(E))$. Let $\Omega \in \mathcal{C}(M)$ with $\Omega \subset E$ satisfy (by \cite[Corollary 2.2]{NR})
    \[ \cH^{n+1}(\Omega) = v\ \ \text{and}\ \ \cH^n(\partial^*\Omega) = \mathscr{I}_E(v). \]

    Let $\Omega_t = \Omega \cup ([t , \infty) \times N)$. Since $\Omega$ has finite perimeter and $M$ has finite volume, we have
    \[ \lim_{t \to \infty} \cH^n(\partial^*\Omega_t) = \cH^n(\partial^*\Omega) \ \ \text{and}\ \ \lim_{t \to \infty} \cH^{n+1}(\Omega_t) = \cH^{n+1}(\Omega_t). \]

    Let $O_t$ be the subset of the form $[\tau_t, \infty) \times N$ so that $\cH^{n+1}(O_t) = \cH^{n+1}(\Omega_t)$.

    Let
    \[ \omega = \rho^n\mathrm{vol}_N. \]
    Then
    \[ d\omega = n\rho^{n-1}\rho'dt \wedge \mathrm{vol}_N = n\frac{\rho'}{\rho\psi}\mathrm{vol}_M. \]
    By Stokes' theorem, we have
    \begin{align*}
        \cH^n(\partial O_t) - \cH^n(\partial^*\Omega_t)
        & \leq \int_{[\partial O_t]} \omega - \int_{[\partial^*\Omega_t]} \omega\\
        & = \int_{[O_t] - [\Omega_t]} d\omega\\
        & = \int_{O_t \setminus \Omega_t}n\frac{\rho'}{\rho\psi} -\int_{\Omega_t \setminus O_t} n\frac{\rho'}{\rho\psi}
    \end{align*}
    Since $O_t$ has the same volume as $\Omega_t$, we have $\cH^{n+1}(O_t \setminus \Omega_t) = \cH^{n+1}(\Omega_t \setminus O_t)$. Moreover, by assumption,
    \[ n\frac{\rho'}{\rho\psi}\Big|_{O_t\setminus \Omega_t} \leq n\frac{\rho'}{\rho\psi}\Big|_{\Omega_t\setminus O_t}. \]
    Hence, we have
    \[ \cH^n(\partial O_t) \leq \cH^n(\partial^*\Omega_t). \]
    
    Taking $t \to \infty$, we have
    \[ \cH^n(\partial O) \leq \cH^n(\partial^*\Omega), \]
    where $O$ is the subset of the form $[\tau, \infty) \times N$ with $\cH^{n+1}(O) = \cH^{n+1}(\Omega) = v$, so
    \[ \mathscr{I}_E(v) = \cH^n(\partial O). \]
    In other words, the areas and volumes of the subsets of the form $[\tau, \infty) \times N$ provide a lower bound for $\mathscr{I}_E$.

    Let $a(t) = \cH^n(\{t\} \times N)$ and $v(t) = \cH^{n+1}([t, \infty) \times N)$. If $t(v)$ denotes the inverse function of $v(t)$, we have
    \[ \eta(v) = a(t(v)) \]
    is a lower bound for $\mathscr{I}_E$. By the coarea formula and the formula for the derivative of an inverse function, we have
    \[ \eta'(v) = \frac{n\rho^{n-1}\rho'}{-\rho^n\psi} = -n\frac{\rho'}{\rho\psi}. \]
    The conclusion follows by Proposition \ref{prop:noncompact_piece}.
\end{proof}

We now compute lower bounds for $\mathscr{D}$ for a few examples.

\begin{example}[Finite volume hyperbolic manifolds]
    Let $(M, g)$ be a complete hyperbolic manifold with finite volume. By the structure theorem for these manifolds (see \cite[Chapter 4.5]{thurston}), there is a compact set $K \subset M$, a constant $A > 0$, and a closed Riemannian manifold $(N, g_N)$ (a union of flat tori and flat Klein bottles) so that $M \setminus K$ is isometric to
    \[ ((A, \infty) \times N, t^{-2}dt^2 \oplus t^{-2}g_N). \]
    In this case, $\psi(t) = \rho(t) = t^{-1}$. We see that
    \[ \rho' = -t^{-2} < 0\ \ \text{and}\ \ \frac{\rho'}{\rho\psi} = -1\ \ \text{is nonincreasing}. \]
    By Proposition \ref{prop:compute}, we have
    \[ \mathscr{D}(M, g) \geq n.\]
    By Theorem \ref{thm:cmc}, we have Corollary \ref{cor:finitevolhyp}.
\end{example}

\begin{example}[$\mathscr{D} = +\infty$]\label{ex:rot_sym}
    Let $(N, g_N)$ be a closed Riemannian manifold. Let $(M, g)$ be a Riemannian manifold isometric to
    \[ ((A, \infty) \times N, dt^2 \oplus e^{-2t^{1+\alpha}}g_N) \]
    outside a compact set, where $\alpha > 0$.
    Completeness follows from construction, and the fact that $M$ has finite volume follows because
    \[ \int_1^{\infty} e^{-nt^{1+\alpha}}dt \leq \int_1^{\infty} e^{-nt}dt = \frac{-1}{n}e^{-nt}\Big|_1^{\infty} = \frac{e^{-n}}{n}. \]
    In this case, we have $\psi(t) = 1$ and $\rho(t) = e^{-t^{1+\alpha}}$. We compute
    \[ \rho' = -(1+\alpha)t^\alpha e^{-t^{1+\alpha}} < 0 \]
    and
    \[ \frac{\rho'}{\rho\psi} = -(1+\alpha)t^\alpha \ \ \text{is nonincreasing}. \]
    By Proposition \ref{prop:compute}, we have
    \[ \mathscr{D}(M, g) \geq \lim_{t \to \infty} n(1+\alpha)t^\alpha = + \infty. \]
    By Theorem \ref{thm:cmc}, we can find a smooth complete almost embedded finite area $c$-CMC hypersurface in $M$ for any $0 < c < \infty$.
\end{example}

\bibliographystyle{amsalpha}
\bibliography{bib}

\end{document}